\documentclass[12pt,a4paper,leqno]{amsart}

\usepackage{amsmath,eucal}
\usepackage{amssymb}

\newtheorem{theorem}{Theorem}[section]

\newtheorem{lemma}[theorem]{Lemma}
\newtheorem{proposition}[theorem]{Proposition}
\newtheorem{cor}[theorem]{Corollary}

\theoremstyle{definition}

\newtheorem{rmk}[theorem]{Remark}
\newtheorem*{Ackn}{Acknowledgment}

\newcommand{\la}{\langle}
\newcommand{\ra}{\rangle}

\newcommand{\dbla}{\langle\!\langle}
\newcommand{\dbra}{\rangle\!\rangle}

\newcommand{\F}{\mathcal F}

\newcommand{\tr}{{\rm tr\,}}
\newcommand{\bA}{{\bf A}}
\newcommand{\bi}{{\bf i}}
\newcommand{\End}{{\rm End\,}}

\title[Generalized Newton transformation and its applications]{Generalized Newton transformation and its applications to extrinsic geometry}

\author[K. Andrzejewski, W. Koz\l owski and K. Niedzia\l omski]{Krzysztof Andrzejewski, Wojciech Koz\l owski\\and Kamil Niedzia\l omski}

\begin{document}
\baselineskip=17pt

\begin{abstract}
In this article we introduce a generalization of the Newton transformation to the case of a system of endomorphisms. We show that it can be used in the context of extrinsic geometry of foliations and distributions yielding new integral formulas containing generalized extrinsic curvatures.
\end{abstract}


\thanks{The authors were supported by the Polish NSC grant No 6065/B/H03/2011/40.} 

\address{
Department of Theoretical Physics and Computer Science \endgraf
University of \L\'{o}d\'{z} \endgraf 
ul. Pomorska 149/153, 90-236 \L\'{o}d\'{z} \endgraf
Poland
}
\email{k-andrzejewski@uni.lodz.pl}

\address{
Department of Mathematics and Computer Science \endgraf
University of \L\'{o}d\'{z} \endgraf
ul. Banacha 22, 90-238 \L\'{o}d\'{z} \endgraf
Poland
}
\email{wojciech@math.uni.lodz.pl}
\email{kamiln@math.uni.lodz.pl}

\maketitle

\section{Introduction} 
Analyzing the study of Riemannian geometry we see that its basic concepts are related with some operators, such as shape, Ricci, Schouten operator, etc. and functions constructed of them, such as mean curvature, scalar curvature, Gauss-Kronecker curvature, etc. The most natural and useful functions are the ones derived from algebraic invariants of these operators e.g. by taking trace, determinant and in general the $r$-th symmetric functions $\sigma_r$. However, the case $r>1$ is strongly nonlinear and therefore more complicated. The powerful tool to deal with this problem is the Newton transformation $T_r$ of an endomorphism $A$ (strictly related with the  Newton's identities) which, in a sense, enables a linearization of $\sigma_r$,
\begin{equation*}
(r+1)\sigma_{r+1}=\tr(AT_r).
\end{equation*}
Although this operator appeared in geometry many years ago (see, e.g., \cite{Rei,Voss}), there is a continues increase of applications of this operator in different areas of geometry in the last years (see, among others, \cite{ AC, AM, ALM,BC,BS,CR,V, Ros,Rov, RW3,Via}).

All these results cause a natural question, what happen if we have a family of operators i.e. how to define the Newton transformation for a family of endomorphisms. A partial answer to this question can be found in the literature (operator $T_r$ and the scalar $S_r$ for even $r$ \cite{AW2,CL}), nevertheless, we expect that this case is much more subtle. This is because in the case of family of operators we should obtain more natural functions as in the case of one and consequently more information about geometry. In order to do this, for any multi--index $u$ and generalized elementary symmetric polynomial $\sigma_u$ we introduce transformations depending on a system of linear endomorphisms. Since these transformations have properties analogous to the Newton transformation (and in the  case of one endomorphism  coincides with it) we call this new object {\it generalized Newton transformation} (GNT) and denote by $T_u$. The concepts of GNT is based on the variational formula for the $r$--th symmetric function
\begin{equation*}
\frac{d}{d\tau}\sigma_{r+1}(\tau)=\tr\left(T_r\cdot \frac{d}{d\tau}A(\tau)\right),
\end{equation*}
which is crucial in many applications and, as we will show,  characterize Newton's transformations. Surprisingly enough, according to knowledge of the authors, GNT has been never investigated before.

The precise definition of GNT and its main properties are given in Sections 2 and 3. These sections seem to be of independent interest since they do not relate to geometric picture. We show some algebraic relations between the trace of GNT and algebraic invariants $\sigma_u$ (Proposition \ref{Prop-GNT}). As a corollary we obtain  generalizations of Cayley--Hamilton theorem (Theorem \ref{Thm-H-C}) for a system of linear endomorphisms. Moreover, we show that the operators $T_r$ and $S_r$ for even $r$, which appeared in the literature, can be build of our operators $T_u$ and $\sigma_u$ (Theorem \ref{thm:relTrTu})

Next, we consider GNT in the context of geometry of foliations (and distributions in general), however we think that GNT has fine algebraic properties, which enable further applications. To begin with, let us note that one of the interesting developments in  geometry of foliations during the last decades was the rise of integral formulas for closed foliated manifolds. These formulas are of some interest, for example in several geometric situations they provide obstructions to the existence of foliations with all the leaves enjoying a given geometric property (see, {\cite{AW1,AW3, BKO,RW3,T} and bibliographies therein). Such formulas have also applications in different areas of differential geometry and analysis on manifolds (see, for example, \cite{BW, BCN, BrSl,Sve}).

The most classical integral formula, in fact the first one known, is due to Reeb \cite{Ree}. He proved that for codimension--one foliation of closed Riemannian manifold $M$ one has
\begin{equation*}
\int_M H =0,
\end{equation*}
where $H$ is the  mean curvature of the leaves. In the early 80's there was obtained a notable result by Brito, Langevin and Rosenberg \cite{BLR}. The authors considered codimension--one foliations of a closed space form $M^{p+1}(\kappa)$.
They showed that the integral of $r$-th basic symmetric function of the shape operator of a foliation $\F$ (i.e. $r$-th mean curvature) depends only on  geometry of $M$ not $\F$. More precisely, they proved that
\begin{equation}\label{Eq}
 \int_MS_r=\left\{ \begin{array}{rl}
\kappa^{\frac{r}{2}}\dbinom{\frac{p}{2}}{\frac{r}{2}}{\rm vol}(M)& \textrm{for $p$, $r$ even},\\ 0 & \textrm{for $p$ or $r$ odd}. 
\end{array} \right.
\end{equation}

A generalization of the above result to the case of arbitrary closed manifold has been recently obtained in \cite{AW1,AW2}. The authors applied the $r$--th Newton transformation $T_r$ of the shape operator of the foliation $\mathcal F$ ($r=0,1,\ldots, p=\textrm{dim}\F$). Computing the divergence of the vector field 
\begin{equation}\label{eq:Ycodim1}
T_r \left(\nabla_N N\right) + S_{r+1} N,
\end{equation}
where $N$ denotes the unit vector field orthogonal to $\mathcal F$, and using the Stokes theorem they obtained system of integral formulas which in the special case of a closed space form reduce to \eqref{Eq}.  

Although, all of the mentioned approaches possess a generalization to the case of arbitrary codimension, that is integral formulas containing higher order mean curvatures $S_r$, for $r$ even, we believe that in codimension grater than one we should have more extrinsic curvatures and, globally defined, (normal) vector fields which can give an additional information about geometry of foliations and distributions.

Since, the bundle $P = O(D^\perp)$ or $P = SO(D^\perp)$ of orthonormal (oriented, respectively) frame fields perpendicular to $D$ codes information on extrinsic geometry of distribution $D$, Section 4 is devoted to the fiber bundle approach to the extrinsic geometry of distributions. Using integration on these bundles we define generalized mean curvatures $\widehat{\sigma_u}$  (see \eqref{eq:sigma_u}) for distributions and total extrinsic  curvatures $\sigma_u^M$ (see \eqref{eq:totalsigma_u}). Moreover, we define a new set of global vector fields $\widehat{Y_u}$ generalizing \eqref{eq:Ycodim1}, obtained from sections $Y_u$ (see \eqref{Eq-Y_u}), by integrating over the fibers of $P$. These fields are crucial in the study of geometry of $D$ and $D^{\bot}$.

In Section 5 we compute the divergence of $Y_u$ and, as a result, we get new integral formulas (Theorem \ref{Thm-main1}) containing $\sigma_u$ together with some terms build of second fundamental form and curvature. 

The next section contains some consequences and  presents our results in some special cases. We  obtain a generalization  of the classical formula obtained by Walczak \cite{Wal} (Corollary \ref{Cor-Walfor}). Moreover, in the case of constant sectional curvature and totally geodesic distribution $D^\perp $ we obtain recurrence formula for $\sigma_u^M$ which implies that it does not depend on geometry of distribution $D$. Using relationships between $\sigma_u$ and $S_r$ we give another proof of the theorem obtained by Brito and Naveira \cite{BN}. Moreover,  we show that when multi-index has only one nonzero element then our formulas reduce to ones obtained in \cite{Rov} and in the case of codimension one to formulas obtained in \cite{AW1}.

Finally, since we could not find suitable references and to make the paper more self--sufficient, Appendix contains proofs of some essential formulas concerning differentiation and integration on principal bundles.

Throughout the paper everything
(manifolds, distribution, foliations, etc.) is assumed to be
smooth and oriented and we will use the following index convention: $n=p+q$ and
\begin{equation*}
i,j,k=1,\ldots,p;\quad\alpha,\beta,\gamma=1,\ldots,q.
\end{equation*}

\section{Generalized Newton transformation (GNT)}

In this section we define and state fundamental properties of Newton transformation associated with an ordered system of endomorphisms. We call these new transformations {\it generalized Newton transformation}. First, we give relevant facts about classical Newton transformations (for more details see \cite{Ros}). 

Let $A$ be an endomorphism of a $p$--dimensional vector space $V$. The {\it Newton transformation} of $A$ is a system $T=(T_r)_{r=0,1,\ldots}$ of endomorphisms of $V$ given by the recurrence relations:
\begin{align*}
T_0 &= 1_V,\\
T_r &= \sigma_r 1_V - A T_{r-1}, \quad r=1,2,\dots 
\end{align*}
Here $\sigma_r$'s are elementary symmetric functions of $A$. If $r>p$ we put $\sigma_r = 0$.
Equivalently, each $T_r$ may be defined by the formula
\begin{equation*} 
T_r = \sum_{j=0}^r (-1)^j\sigma_{r-j} A^j.
\end{equation*}
Observe that $T_p$ is the characteristic polynomial of $A$. Consequently, by Hamilton--Cayley Theorem $T_p=0$. It follows that $T_r=0$ for all $r\geq p$. 

The Newton transformation satisfies the following relations \cite{Rei}:   
\begin{enumerate}
\item[(N1)] Symmetric function $\sigma_r$ is given by the formula
\begin{equation*}
r\sigma_r=\tr(AT_{r-1}).
\end{equation*} 
\item[(N2)] Trace of $T_r$ is equal 
\begin{equation*} 
\tr T_r = (p-r) \sigma_r.
\end{equation*}
\item[(N3)] If $A(\tau)$ is a smooth curve in $\End(V)$ such that $A(0)=A$, then
\begin{equation*}
\frac{d}{d\tau}\sigma_{r+1}(\tau)_{\tau=0}=\tr (\frac{d}{d\tau}A(\tau)_{\tau=0}\cdot T_r),\quad r=0,1,\ldots,p.
\end{equation*}
\end{enumerate}

Condition (N3) is the starting point to define generalized Newton transformations.

Let $V$ be a $p$--dimensional vector space (over $\mathbb R$) equipped with an inner product $\la\,,\ra$. For an endomorphism $A\in \End (V)$, let $A^{\top}$ denote the adjoint endomorphism, i.e. $\la Av,w\ra = \la v,A^{\top} w\ra$  for every $v,w\in V$. The space
$\End (V)$ is equipped with an inner product 
\begin{equation*}
\dbla A, B\dbra = \tr (A^{\top} B), \quad A,B \in \End (V).
\end{equation*}

Let $\mathbb{N}$ denote the set of nonnegative integers. By $\mathbb N (q)$ denote the set of all sequences $u=(u_1,\ldots, u_q)$, with $u_j\in \mathbb N$. The length $|u|$ of $u\in \mathbb N (q)$ is given by $|u| = u_1+\ldots + u_q$. Denote by $\End^q(V)$ the vector space $\End(V)\times \ldots \times \End(V)$ ($q$--times). For $\bA=(A_1,\ldots,A_q) \in \End^q(V)$, $t=(t_1,\ldots,t_q)\in \mathbb R^q$ and $u\in \mathbb N (q)$ put
\begin{align*}
t^u &= t_1^{u_1}\ldots t_q^{u_q},\\
t \bA &= t_1  A_1 + \ldots + t_q A_q
\end{align*}

By a {\it Newton polynomial} of $\bA$ we mean a polynomial $P_{\bA}:\mathbb{R}^q\to\mathbb{R}$ of the form
$P_{\bA} (t) = \det (1_V + t \bA)$. Expanding $P_{\bA}$ we get
\begin{equation*}
P_{\bA}(t) = \sum_{|u|\leq p} \sigma_u t^u,
\end{equation*}
where the coefficients $\sigma_u = \sigma_u(\bA)$ depend only  on $\bA$. Observe that $\sigma_{(0,\ldots,0)}=1$. It is convenient to put $\sigma_u = 0$ for $|u|>p$. 

Consider the following  (music) convention. For $\alpha$ we define
functions $\alpha^\sharp: \mathbb{N}(q)\to \mathbb N(q)$ and $\alpha_\flat: \mathbb{N}(q)\to \mathbb N(q)$ as follows
\begin{align*}
\alpha^\sharp (i_1,\ldots,i_q) &=(i_1,\ldots,i_{\alpha-1},i_\alpha+1,i_{\alpha+1},\ldots, i_q),\\
\alpha_\flat(i_1,\ldots,i_q) &=(i_1,\ldots,i_{\alpha-1},i_\alpha-1,i_{\alpha+1},\dots, i_q),
\end{align*}
i.e. $\alpha^\sharp$ increases the value of the $\alpha$--th element by $1$ and $\alpha_\flat$ decreases the value of $\alpha$--th element by $1$. It is clear that $\alpha^\sharp$ is the inverse map to $\alpha_\flat$.

Now, we may state the main definition. The {\it generalized Newton transformation} of $\bA=(A_1,\ldots,A_q)\in\End^q(V)$ is a system of endomorphisms $T_u=T_u(\bA)$, $u\in \mathbb N (q)$, satisfying the following condition (generalizing (N3)):

For every smooth curve $\tau \mapsto \bA(\tau)$ in $\End^q(V)$ such that $\bA (0) =  \bA$
\begin{equation}\label{Eq_GNT_sigma}\tag{GNT}
\begin{split}
\frac{d}{d\tau}\sigma_u(\tau)_{\tau=0} 
&= \sum_{\alpha} \dbla \frac{d}{d\tau}A_\alpha(\tau)_{\tau=0})^{\top} | T_{\alpha_\flat(u)}\dbra\\
&=\sum_{\alpha}\tr\left( \frac{d}{d\tau}A_{\alpha}(\tau)_{\tau=0}\cdot T_{\alpha_\flat(u)} \right).
\end{split}
\end{equation}

From the above definition it is not clear that generalized Newton transformation exists. In order to show the existence of Generalized Newton transformation, we introduce the following notation.

For $q,s\geq 1$ let $\mathbb{N}(q,s)$ be the set of all $q\times s$ matrices, whose entries are elements of $\mathbb N$. Clearly, the set $\mathbb{N}(1,s)$ is the set of multi--indices $i=(i_1,\ldots, i_s)$ with $i_1,\ldots,i_s\in \mathbb N$, hence $\mathbb{N}(s)=\mathbb{N}(1,s)$. Moreover, every matrix $\bi =(i^{\alpha}_l)\in \mathbb{N}(q,s)$ may be identified with an ordered system $\bi = (i^1,\ldots,i^q)$ of multi--indices $i^{\alpha} = (i^{\alpha}_1,\ldots,i^{\alpha}_s)$.

If $i=(i_1,\ldots,i_s) \in \mathbb{N}(s)$ then its {\it length} is simply the number $|i| = i_1+\ldots+i_s$. For $\bi = (i^1,\ldots,i^q)\in \mathbb{N}(q,s)$ we define its {\it weight} as an multi--index $|\bi| = (|i^1|,\ldots,|i^q|)\in \mathbb{N}(q)$. By the {\it length} $\| \bi\|$ of $\bi$ we mean the length of $|\bi|$, i.e., $\|\bi\| = \sum_{\alpha} |i^{\alpha}| = \sum_{\alpha,l} i^{\alpha}_l$.

Denote by $\mathbb{I}(q,s)$ a subset of $\mathbb N(q,s)$ consisting of all matrices $\bi$ satisfying the following conditions:
\begin{enumerate}
\item every entry of $\bi$ is either $0$ or $1$,
\item the length of $\bi$ is equal to $s$,
\item in every column of $\bi$ there is exactly one entry equal to $1$, or equivalently 
$|\bi^{\top}| = (1,\ldots,1)$.
\end{enumerate}
We identify $\mathbb{I}(q,0)$ with a set consisting of the zero vector $0=[0,\ldots,0]^{\top}$.

Let $\bA\in \End^q(V)$, $\bA = (A_1,\dots,A_q)$, and $\bi\in\mathbb{N}(q,s)$. By  $\bA^\bi$ we mean an endomorphism (composition of endomorphisms) of the form
\begin{equation*} 
\bA^\bi = A_1^{i^1_1} A_2^{i^2_1}\ldots A_q^{i^q_1} A_1^{i^1_2}\ldots 
A_q^{i^q_2}\ldots A_1^{i^1_s}\ldots A_q^{i^q_s}.
\end{equation*}
In particular, $\bA^0 = 1_V$.

\begin{theorem}\label{thm:existenceGNT}
For every system of endomorphisms $\bA=(A_1,\dots,A_q)$, there exists unique generalized Newton transformation $T= (T_u: u\in \mathbb N (q))$ of $\bA$. Moreover, each $T_u$ is given by the formula
\begin{equation}\label{Eq--Thm--GNF}
T_u  =\sum_{s=0}^{|u|} \sum_{\bi \in \mathbb{I}(q,s)}
(-1)^{\|\bi\|}\sigma_{u-|\bi|}  \bA^{\bi},
\end{equation}
where \(\sigma_{u-|\bi|}  = \sigma_{u-|\bi|} (\bA)\). 
\end{theorem}

The proof will be divided into steps. The following two technical lemmas are well known.

\begin{lemma}\label{lemma--1}
Let $\bA \in  \End^q(V)$. There exists $\varepsilon >0$ such that for every $t\in \mathbb R^q$ with $|t| < \varepsilon$, $1_V + t \bA$ is an isomorphism of $V$ and its inverse is given by the formula
\begin{equation*} 
(1_V + t \bA)^{-1} =  \sum_{s=0}^\infty (-1)^s \sum_{\bi \in \mathbb{I}(q,s)} t^{|\bi|} A^{\bi}.
\end{equation*} 
\end{lemma}

\begin{lemma}\label{lemma--2}
If $\tau \mapsto A(\tau)$ is a smooth curve in $\End(V)$ such that $A(0) = 1_V$, then 
\begin{equation}\label{ForCtau}
\frac {d}{d \tau}\left( \det A(\tau)\right)_{\tau=0} = \tr \left(\frac{d}{d \tau}A(\tau)_{\tau=0}\right).
\end{equation}
\end{lemma}

Moreover, we have the following result.

\begin{proposition}\label{Cor--1} 
Consider a curve $\tau\mapsto \bA (\tau)$ in $\End^q(V)$. Put $\bA(0) = \bA$ and $\bA '= \frac{d}{d\tau} \bA(\tau)_{\tau=0}$. Then there exists $\varepsilon >0$ such that for every $t\in\mathbb{R}^q$ with $|t|<\varepsilon$, we have 
\begin{equation}\label{Eq--Cor--1}
\frac{d}{d\tau}P_{\bA(\tau)}(t)_{\tau=0} = \tr \left(
 t\bA ' \sum_{s=0}^\infty (-1)^s \sum_{\bi \in \mathbb{I}(q,s)} t^{|\bi|} \bA^{\bi}
\right) P_{\bA }(t).
\end{equation}
\end{proposition}

\begin{proof} By Lemma \ref{lemma--1}, there exists $\varepsilon >0$ such that for every $t\in \mathbb R^q$ with $|t| < \varepsilon$ the endomorphism $1_V + t \bA$ is invertible. For fixed $t\in \mathbb R^q$, $|t|<\varepsilon$, consider a curve $B(\tau) = (1_V+t\bA(\tau))(1_V + t \bA)^{-1}$. Clearly $B$ is smooth and satisfies the assumptions of Lemma \ref{lemma--2}. Moreover
\begin{equation*} 
\det B(\tau) = \frac{P_{\bA(\tau)}(t)}{ P_{\bA}(t)}
\end{equation*}
and the denominator in the above fraction does not depend on $\tau$. Therefore, applying Lemma \ref{lemma--2}, we have
\begin{equation*}
\frac{d}{d\tau}P_{\bA(\tau)}(t)_{\tau=0} = \tr \left(\frac{d}{d \tau}B(\tau)_{\tau=0}\right) P_{\bA}(t).
\end{equation*}
On the other hand, applying Lemma \ref{lemma--1}, one can get 
\begin{equation*}
\frac{d}{d \tau}B(\tau)_{\tau=0} = t\bA ' \sum_{s=0}^\infty (-1)^s \sum_{\bi \in \mathbb{I}(q,s)} t^{|\bi|} \bA^{\bi}.
\end{equation*}
Combining these two equalities lemma holds.
\end{proof}

\begin{proof}[Proof of Theorem \ref{thm:existenceGNT}]

{\it Existence}. Let $\tau\mapsto \bA(\tau)$ be a curve in $\End^q(V)$ such that $\bA(0)=\bA$. By Proposition \ref{Cor--1} there exists $\varepsilon>0$ such that for every $t\in\mathbb{R}^q$ with $|t|<\varepsilon$, \eqref{Eq--Cor--1} holds. Denote by $L$ and $R$ the left hand and the right hand side of \eqref{Eq--Cor--1}, respectively. Then
\begin{align*}
R &= \tr \left( t\frac{d}{d\tau}\bA(\tau)_{\tau=0} \sum_{s=0}^\infty (-1)^s \sum_{\bi \in \mathbb{I}(q,s)} t^{|\bi|} \bA^{\bi}\right) P_{\bA }(t)\\
&= \tr \left( t\frac{d}{d\tau}\bA(\tau)_{\tau=0} \sum_{s=0}^\infty (-1)^s \sum_{\bi \in \mathbb{I}(q,s)} t^{|\bi|} \bA^{\bi}\right) \sum_{|a|=0 }^\infty 
\sigma_a t^a\\
&= \tr \left( t\frac{d}{d\tau}\bA(\tau)_{\tau=0} \sum_{s=0}^\infty 
\sum_{|a|=0 }^\infty  \sum_{\bi \in \mathbb{I}(q,s)} t^{|\bi|+a} (-1)^s\sigma_a
 \bA^{\bi}\right)\\
&= \sum_{\alpha} \tr \left( \frac{d}{d\tau}A_{\alpha}(\tau)_{\tau=0} \sum_{s=0}^\infty 
\sum_{|a|=0 }^\infty  \sum_{\bi \in \mathbb{I}(q,s)} t^{\alpha^\sharp(|\bi|+a)} (-1)^s\sigma_a
 \bA^{\bi}\right).
\end{align*}
Put $u = \alpha^\sharp(|\bi| +a)\in\mathbb{N}(q)$. Then $\alpha_\flat (u) = |\bi|+a$. Hence
\begin{equation*}
R=\sum_{|u|=1}^\infty\sum_{\alpha}\tr\left( \frac{d}{d\tau}A_{\alpha}(\tau)_{\tau=0} 
\sum_{s=0}^{|u|}
\sum_{\bi \in \mathbb{I}(q,s)}
(-1)^{\|\bi\|}\sigma_{\alpha_\flat(u)-|\bi|}  \bA^{\bi}\right)t^u.
\end{equation*}
On the other hand 
\begin{equation*}
P_{\bA(\tau)}(t) = \sum_{|u|=0}^\infty \sigma_u(\tau) t^u. 
\end{equation*}
Thus 
\begin{equation*} 
L= \sum_{|u|=0}^\infty \left(\frac{d}{d\tau}\sigma_u(\tau)_{\tau=0}\right) t^u.
\end{equation*}
Since $L=R$ for every $|t|<\varepsilon$, comparing appropriate monomials, we get 
\begin{equation*}
\frac{d}{d\tau}\sigma_u(\tau)_{\tau=0}=\sum_{\alpha} \tr\left(\frac{d}{d\tau}A_{\alpha}(\tau)_{\tau=0} T_{\alpha_\flat(u)}\right),
\end{equation*}
where $T_u$'s are given by \eqref{Eq--Thm--GNF}. Hence, the system $T=(T_u, u \in \mathbb{N} (q))$ is the generalized Newton transformation of $\bA$. 

{\it Uniqueness}. Suppose $(S_u, u\in \mathbb N (q))$ is another generalized Newton transformation of $\bA$. We will show that $T_u = S_u$, for every multi--index $u$. Consider a curve $\bA(\tau) = (A_1 + \tau(T_u - S_u)^{\top},A_2,\dots, A_q)$. Then
$\frac{d}{d\tau}\bA(\tau)_{\tau=0} = ((T_u - S_u)^{\top},0,\dots,0)$. Since $T_u$ and $S_u$ are generalized Newton transformation, putting $\alpha=1$ in \eqref{Eq_GNT_sigma}, we get
$\dbla T_u - S_u | T_u\dbra = \frac{d}{d\tau}\sigma_u(\tau)_{\tau=0} = \dbla T_u - S_u | S_u \dbra$. Therefore, $\dbla T_u - S_u | T_u - S_u  \dbra = 0$. Since $\dbla, \dbra$ is an inner product, $T_u = S_u$.
\end{proof}

At the end of this section we want to compare generalized Newton transformation with the one considered in the literature introduced by Reilly \cite{Rei2} and considered further, for example, by Cao and Li \cite{CL}. These transformations, $T_r$ and $T^{\alpha}_{r+1}$, are defined for $r$ even and $\alpha=1,\ldots,q$. Namely, in coordinates 
\begin{equation*}
(T_r)_{ij}=\frac{1}{r!}\sum_{\stackrel{i_1,\ldots,i_r}{ j_1,\ldots,j_r}}\delta^{i_1,\ldots,i_r,i}_{j_1,\ldots,j_r,j}\sum_{\alpha_1,\ldots,\alpha_{\frac{r}{2}}}
(A_{\alpha_1})_{i_1j_1}(A_{\alpha_1})_{i_2j_2}\ldots(A_{\alpha_{\frac{r}{2}}})_{i_{r-1}j_{r-1}}
(A_{\alpha_{\frac{r}{2}}})_{i_rj_r}
\end{equation*}
and
\begin{multline*}
(T_{r+1}^{\alpha})_{ij}=\frac{1}{r!}\sum_{\stackrel{i_1,\ldots,i_{r+1}}{ j_1,\ldots,j_{r+1}}}\delta^{i_1,\ldots,i_{r+1},i}_{j_1,\ldots,j_{r+1},j}(A_{\alpha})_{i_{r+1}j_{r+1}} \\ 
\cdot\sum_{\alpha_1,\ldots,\alpha_{\frac{r}{2}}}
(A_{\alpha_1})_{i_1j_1}(A_{\alpha_1})_{i_2j_2}\ldots(A_{\alpha_{\frac{r}{2}}})_{i_{r-1}j_{r-1}}
(A_{\alpha_{\frac{r}{2}}})_{i_rj_r},
\end{multline*}
where $\delta^{i_1\ldots i_r}_{j_1\ldots j_r}$ is the generalized Kronecker symbol, which is $+1$ or $-1$ according as the $i$'s are distinct and the $j$'s are an even or odd permutation of the $i$'s, and which is $0$ in all other cases. Moreover, we define functions $S_r$ for $r$ even in the following way
\begin{equation*}
S_r=\frac{1}{r!}\sum_{\stackrel{i_1,\ldots,i_r}{ j_1,\ldots,j_r}}\delta^{i_1,\ldots,i_r}_{j_1,\ldots,j_r}\sum_{\alpha_1,\ldots,\alpha_{\frac{r}{2}}}
(A_{\alpha_1})_{i_1j_1}(A_{\alpha_1})_{i_2j_2}\ldots(A_{\alpha_{\frac{r}{2}}})_{i_{r-1}j_{r-1}}
(A_{\alpha_{\frac{r}{2}}})_{i_rj_r}.
\end{equation*}

These transformations satisfy the same relations as Newton transformations. Namely, for $r$ even, we have \cite{CL,AW1}
\begin{enumerate}
\item[(R1)] $T_r=S_r 1-\sum_{\alpha}T_{r-1}^{\alpha}A_{\alpha}$, $T_0=1$,
\item[(R2)] $\tr T_r=(p-r)S_r$,
\item[(R3)] $\frac{d}{d\tau}\left(S_r(\tau)\right)_{\tau=0}=
\sum_{\alpha}\tr\left(\frac{d}{d\tau}A_{\alpha}(\tau)_{\tau=0}\cdot T_{r-1}^{\alpha}\right)$, where $A_{\alpha}(\tau)$ is a curve such that $A_{\alpha}(0)=A_{\alpha}$.
\end{enumerate} 
Moreover, condition $(R3)$ is equivalent to the definition of $T_{r-1}^{\alpha}$.

It turns out that these transformations are linear combinations of generalized Newton transformation $T_u$. First adopt the following notation. For a multi--index $u\in\mathbb{N}(q)$ of length $r$ let
\begin{equation*}
u!=u_1!u_2!\ldots u_q!\quad\textrm{and}\quad \binom{r}{u}=\frac{r!}{u!}=\frac{r!}{u_1!u_2!\ldots u_q!}.
\end{equation*}
Let $2\mathbb{N}(q)$ denote the set of all multi--indices $u\in\mathbb{N}(q)$ such that each $u_1,\ldots,u_q$ is even.

\begin{theorem}\label{thm:relTrTu}
For $r$ even
\begin{equation}\label{eq:relTrTu}
T_r=\sum_{\stackrel{u\in 2\mathbb{N}(q)}{|u|=r}} \binom{\frac{r}{2}}{\frac{u}{2}}\binom{r}{u}^{-1}\, T_u,\quad 
T_{r+1}^{\alpha}=\sum_{\stackrel{u\in 2\mathbb{N}(q)}{|u|=r}} \binom{\frac{r}{2}}{\frac{u}{2}}\binom{r}{u}^{-1}\, T_{\alpha_\flat(u)}
\end{equation}
and
\begin{equation}\label{eq:relSrsigmau}
S_r=\sum_{\stackrel{u\in 2\mathbb{N}(q)}{|u|=r}} \binom{\frac{r}{2}}{\frac{u}{2}}\binom{r}{u}^{-1}\,\sigma_u,
\end{equation}
where $\frac{u}{2}=(\frac{u_1}{2},\ldots,\frac{u_q}{2})$.
\end{theorem}

Before we turn to the proof of \eqref{eq:relTrTu} and \eqref{eq:relSrsigmau} recall the properties of the generalized Kronecker symbol. One can show that
\begin{equation}\label{eq:Kroneckersym1}
\sum_{i_1,\ldots,i_r}\delta^{i_1\ldots i_r}_{i_1\ldots i_r}=\frac{p!}{(p-r)!}
\end{equation}
and
\begin{equation}\label{eq:Kroneckersym2}
\sum_{i_1,\ldots,i_s}\delta^{i_1\ldots i_r i_{r+1}\ldots i_s}_{j_1\ldots j_r i_{r+1}\ldots i_s}=\frac{(p-s)!}{(p-r)!}\sum_{i_1,\ldots,i_r}\delta^{i_1\ldots i_r}_{j_1\ldots j_r}.
\end{equation} 

Now we are able to derive the exact formula for $\sigma_u=\sigma_u(A_1,\ldots,A_q)$.

\begin{proposition}\label{prop:formulasigmau}
For any indices $\alpha_1,\ldots,\alpha_r$ we have
\begin{equation}\label{eq:explicitsigmau}
\sigma_{\alpha_1^\sharp\ldots\alpha_r^\sharp(0,\ldots,0)}=\frac{1}{u!}\sum_{\stackrel{i_1,\ldots,i_r}{ j_1,\ldots,j_r}}\delta^{i_1,\ldots,i_r}_{j_1,\ldots,j_r}(A_{\alpha_1})_{i_1j_1}\ldots(A_{\alpha_r})_{i_rj_r},
\end{equation}
where $u=\alpha_1^\sharp\ldots\alpha_r^\sharp(0,\ldots,0)$.
\end{proposition}
\begin{proof}
By the definition of symmetric functions $\sigma_u$ we have
\begin{equation*}
\det(1+t\bA)=\sum_{|u|\leq p}\sigma_ut^u=\sum_{|u|\leq p}\sigma_ut_1^{u_1}\ldots t_q^{u_q}.
\end{equation*}
On the other hand, by the formula for the determinant
\begin{equation*}
\det A=\frac{1}{p!}\sum_{\stackrel{i_1,\ldots,i_p}{ j_1,\ldots,j_p}}\delta^{i_1,\ldots,i_p}_{j_1,\ldots,j_p}A_{i_1j_1}\ldots A_{i_pj_p}
\end{equation*} 
we have
\begin{align*}
\det(1+t\bA) &=\frac{1}{p!}\sum_{\stackrel{i_1,\ldots,i_p}{ j_1,\ldots,j_p}}\delta^{i_1,\ldots,i_p}_{j_1,\ldots,j_p}\Big(\delta^{i_1}_{j_1}
+\sum_{\alpha_1}t_{\alpha_1}(A_{\alpha_1})_{i_1j_1}\Big)\ldots \Big(\delta^{i_p}_{j_p}+\sum_{\alpha_p}t_{\alpha_p}(A_{\alpha_p})_{i_pj_p}\Big).
\end{align*}
It follows that by \eqref{eq:Kroneckersym2}
\begin{equation*}
\sigma_{\alpha_1^\sharp\ldots\alpha_r^\sharp(0,\ldots,0)}=\frac{c}{p!}\sum_{\stackrel{i_1,\ldots,i_r}{ j_1,\ldots,j_r}}\delta^{i_1,\ldots,i_r}_{j_1,\ldots,j_r}(A_{\alpha_1})_{i_1j_1}\ldots(A_{\alpha_r})_{i_rj_r}
\end{equation*}
for a constant $c$ independent of the system of matrices $A_1,\ldots,A_q$. Hence, taking $A_1=\ldots=A_q=1$ by \eqref{eq:Kroneckersym1} we get
\begin{equation}\label{eq:csigmau1}
\sigma_{\alpha_1^\sharp\ldots\alpha_r^\sharp(0,\ldots,0)}(1,\ldots,1)=
\frac{c}{p!}\sum_{i_1,\ldots,i_r}\delta^{i_1,\ldots,i_r}_{i_1,\ldots,i_r}=\frac{c}{(p-r)!}
\end{equation}
Moreover, using multinomial theorem,
\begin{equation*}
\det ((1+t_1+\ldots+t_q)1)=\sum_{r=0}^p\sum_{u_1+\ldots+u_q=r}\frac{p!}{(p-r)!u!}t_1^{u_1}\ldots t_q^{u_q}.
\end{equation*}
Therefore
\begin{equation}\label{eq:csigmau2}
\sigma_{(u_1,\ldots,u_q)}(1,\ldots,1)=\frac{p!}{(p-r)!u!}.
\end{equation}
By \eqref{eq:csigmau1} and \eqref{eq:csigmau2} we have $\frac{c}{p!}=\frac{1}{u!}$, hence \eqref{eq:explicitsigmau} holds.
\end{proof}

\begin{proof}[Proof of Theorem \ref{thm:relTrTu}]
Any multi--index $u\in 2\mathbb{N}(q)$ of length $r$ is of the form $u=(\alpha_1^\sharp)^2\ldots(\alpha_{\frac{r}{2}}^\sharp)^2(0,\ldots,0)$ for some indices $\alpha_1,\ldots,\alpha_{\frac{r}{2}}$. For such a multi--index, by Proposition \ref{prop:formulasigmau}, we have
\begin{equation*}
\sigma_u=\frac{1}{u!}\sum_{\stackrel{i_1,\ldots,i_r}{ j_1,\ldots,j_r}}\delta^{i_1,\ldots,i_r}_{j_1,\ldots,j_r}
(A_{\alpha_1})_{i_1j_1}(A_{\alpha_1})_{i_2j_2}\ldots(A_{\alpha_{\frac{r}{2}}})_{i_{r-1}j_{r-1}}(A_{\alpha_{\frac{r}{2}}})_{i_rj_r}
\end{equation*}
On the other hand, notice that there are $\binom{\frac{r}{2}}{\frac{u}{2}}$ indices $\alpha_1,\ldots,\alpha_{\frac{r}{2}}$ which give $u$. Thus \eqref{eq:relSrsigmau} holds. Relations \eqref{eq:relTrTu} follow immediately from the properties (R1) and (R3) of transformations $T_r$ and $T_{r-1}^{\alpha}$.
\end{proof}

\section{Properties of generalized Newton Transformation}

\begin{theorem}[Generalized Hamilton--Cayley Theorem]\label{Thm-H-C} Let $T = (T_u: u\in \mathbb N (q))$ be the generalized Newton transformation of $\bA$. Then for every $u \in \mathbb N (q)$ of length greater or equal to $p$ we have $T_u = 0$. 
\end{theorem}
\begin{proof} 
Fix a multi--index $u=(u_1,\ldots,u_q)$ with $|u|\geq p$. We may assume that $u_1>0$. Take $\alpha=1$. Consider the curve $\bA = (A_1 + \tau T_{\alpha^\sharp(u)}^{\top},0,\ldots,0)$.
Then $\frac{d}{d\tau}\bA(\tau)_{\tau=0}= (T_{\alpha^\sharp(u)}^{\top},0,\ldots,0)$. Moreover, since $|u|\geq p$, then $|\alpha^\sharp (u)|>p$.  Therefore, $\sigma_{\alpha^\sharp (u)}(\tau) =0$ for every $\tau$. Consequently, by the definition of generalized Newton transformation we obtain
\begin{equation*}
0= \frac{d}{d\tau}\sigma_{\alpha^\sharp (u)}(\tau)_{\tau=0}=\dbla T_u|  T_u\dbra.
\end{equation*}
Since \(\dbla,\dbra\) is the inner product, we get $T_u=0$.
\end{proof}

\begin{theorem}\label{Thm-Formula-GNT} 
The generalized Newton transformation $T=(T_u: u\in \mathbb N (q))$ of $\bA$ satisfies the following recurrence relations:
\begin{align}
T_0 &= 1_V, &&\textrm{where $0=(0,\ldots, 0)$},\label{Eq-T_0}\\
\begin{split}\label{Eq-T_u}
T_u &= \sigma_u 1_V - \sum_{\alpha} A_\alpha T_{\alpha_\flat (u)} \\
&= \sigma_u 1_V - \sum_{\alpha}T_{\alpha_\flat (u)}  A_\alpha , 
\end{split}
&& \textrm{where $|u|\geq 1$}.
\end{align}
\end{theorem}
\begin{proof}
Equality \eqref{Eq-T_0} is obvious. Assume $u\in \mathbb N (q)$ and $|u|=s+1$, $s\geq 0$.  We show the first identity of \eqref{Eq-T_u}. A proof of the second one is analogous. 

For every matrix $\bi = (i^{\alpha}_l) \in \mathbb{I}(q,s)$ and every $\beta$ define a matrix $\beta \circ \bi \in \mathbb{I}(q,s+1)$ by 
\begin{equation*} 
(\beta \circ \bi)^{\alpha}_1 = 
\delta^{\alpha}_\beta,\quad (\beta \circ \bi)^{\alpha}_l = i^{\alpha}_{l-1},\quad 2\leq l\leq s+1.
\end{equation*}
It is easy to observe that $\mathbb{I}(q,s+1)$ can be expressed as the following sum of pairwise disjoint sets 
\begin{equation*}
\mathbb{I}(q,s+1) = \bigcup_{\beta} \beta\circ \mathbb{I}(q,s).
\end{equation*}
Moreover, for every $\bi \in \mathbb{I}(q,s)$, $|\beta \circ \bi| = \beta^\sharp (|\bi|)$ and
$A^{\beta\circ \bi} = A_\beta A^\bi$. We have 
\begin{align*}
T_u &= \sum_{l=0}^{s+1} \sum_{\bi \in \mathbb{I}(q,l)} (-1)^{\|\bi\|} \sigma_{u-|\bi|} A^\bi\\
&= \sigma_u 1_V + \sum_{l=1}^{s+1} \sum_{\beta} \sum _{\bi \in \mathbb{I}(q,l-1)} 
(-1)^{\|\beta \circ \bi\|} \sigma_{u-|\beta \circ \bi|} A^{\beta \circ \bi}\\
&= \sigma_u 1_V+ \sum_{l=0}^s \sum_{\beta} \sum _{\bi \in \mathbb{I}(q,l)} (-1)^{1+ \|\bi\|} 
\sigma_{\beta_\flat (u) - |i|} A_\beta A^{\bi}\\
&= \sigma_u 1_V- \sum_{\beta} A_\beta\left( \sum_{l=0}^s \sum _{\bi \in \mathbb{I}(q,l)}
\sigma_{\beta_\flat (u) - |i|} A^{\bi}\right)\\
&= \sigma_u 1_V- \sum_{\beta} A_\beta T_{\beta_\flat (u)},
\end{align*}
for $s= |\beta_\flat(u)|$.
\end{proof}

Now we state and prove the fundamental properties of generalized Newton transformation.

\begin{proposition}\label{Prop-GNT}
The generalized Newton transformation $T=(T_u:u\in\mathbb{N}(q))$ of $\bA$ satisfies the following conditions:
\begin{enumerate}
\item[(GN1)] Symmetric functions $\sigma_u$ are given by the formula
\begin{equation*}
|u|\sigma_u=\sum_{\alpha}\tr(A_{\alpha}T_{\alpha_\flat(u)}).
\end{equation*}
\item[(GN2)] The trace of $T_u$ equals
\begin{equation*}
\tr T_u=(p-|u|)\sigma_u.
\end{equation*}
\item[(GN3)] Symmetric functions $\sigma_u$ satisfy the following recurrence relation
\begin{equation*}
\sum_{\alpha,\beta}\tr(A_{\alpha}A_{\beta}T_{\beta_\flat\alpha_\flat(u)})=-|u|\sigma_u
+\sum_{\alpha}(\tr A_{\alpha})\sigma_{\alpha_\flat(u)}. 
\end{equation*}
\end{enumerate}
\end{proposition}
\begin{proof}
{\it Proof of} (GN1). Consider a curve $\bA(\tau)=(1+\tau)\bA$. Then $\bA(0)=\bA$ and $\frac{d}{d\tau}\bA(\tau)_{\tau=0}=\bA$. Expanding the polynomial $P_{\bA(\tau)}(t)$ it is easy to see that $\sigma_u(\tau)=(1+\tau)^{|u|}\sigma_u$, where $\sigma_u$ is a symmetric function of $\bA$. Thus
\begin{equation*}
\frac{d}{d\tau}\sigma_u(\tau)_{\tau=0}=|u|\sigma_u.
\end{equation*}  
Hence by \eqref{Eq_GNT_sigma}
\begin{equation*}
|u|\sigma_u=\sum_{\alpha}\tr(A_{\alpha}T_{\alpha_\flat(u)}).
\end{equation*}

{\it Proof of} (GN2). Follows directly by \eqref{Eq-T_u} and (GN1).

{\it Proof of} (GN3). By \eqref{Eq-T_u} we have
\begin{equation*}
T_{\beta_\flat(u)}=\sigma_{\beta_\flat(u)}1_V-\sum_{\alpha}A_{\alpha}T_{\beta_\flat\alpha_\flat(u)}.
\end{equation*}
Therefore
\begin{align*}
T_u &=\sigma_u 1_V-\sum_{\beta}\left( \sigma_{\beta_\flat(u)}1_V-\sum_{\alpha}A_{\alpha}T_{\beta_\flat\alpha_\flat(u)} \right)\\
&=\sigma_u 1_V-\sum_{\beta}\sigma_{\beta_\flat(u)}A_{\beta}+
\sum_{\alpha,\beta}A_{\beta}A_{\alpha}T_{\beta_\flat\alpha_\flat(u)}.
\end{align*}
Taking the trace, using (GN2) and the fact that $\alpha_\flat$ and $\beta_\flat$ commute we get (GN3).
\end{proof}

\begin{cor}\label{Cor-adj-T_u}
Let $T = (T_u: u\in \mathbb N (q))$ be the generalized Newton transformation of $\bA$.
If every matrix $A_{\alpha}$ is self--adjoint, i.e. $A_{\alpha}^{\top}= A_{\alpha}$, then $T_u$ is self--adjoint for every $u\in \mathbb{N}(q)$. 
\end{cor}

\begin{proof} We apply induction with respect to $k= |u|$. Obviously,  if $k=0$, then $u=0$ and $T_0$ is self--adjoint. Take $k\geq 1$ and assume $T_w$ is self--adjoint for every multi--index $w$ of length $|w|=k-1$. Let $u$ be of length $k$.  Applying \eqref{Eq-T_u} we get
\begin{equation*}
T^{\top}_u = \sigma_u 1_V - \sum_{\alpha}  T^{\top}_{\alpha_\flat (u)} A^{\top}_\alpha
= \sigma_u 1_V - \sum_{\alpha}  T_{\alpha_\flat (u)} A_\alpha
= T_u.
\end{equation*}
Induction completes the proof.
\end{proof}

\section{Extrinsic curvatures for distributions of arbitrary codimension}

Let $(M,g)$ be an oriented Riemannian manifold, $D$ a $p$--dimensional (transversally oriented) distribution on $M$. Let $q$ denotes the codimension of $D$. For each $X\in T_xM$ there is unique decomposition
\begin{equation*}
X=X^{\top}+X^{\bot},
\end{equation*}
where $X^{\top}\in D_x$ and $X^{\bot}$ is orthogonal to $D_x$. Denote by $D^{\bot}$ the bundle of vectors orthogonal to $D$. Let $\nabla$ be the Levi--Civita connection of $g$. $\nabla$ induces connections $\nabla^{\top}$ and $\nabla^{\bot}$ in vector bundles $D$ and $D^{\bot}$ over $M$, respectively. 

Let $\pi:P\to M$ be the principal bundle of orthonormal frames (oriented orthonormal frames, respectively) of $D^{\bot}$. Clearly, the structure group $G$ of this bundle is $G=O(q)$ ($G=SO(q)$, respectively). We define a Riemannian metric on $P$ by inducing the metric from $M$ and an invariant inner product $\dbla\,,\,\dbra$ on the Lie algebra $\mathfrak{g}$ of $G$, 
\begin{equation*}
\dbla A,B\dbra =-\tr (AB),\quad A,B\in \mathfrak{g}.
\end{equation*}
In particular, the projection $\pi:P\to M$ is a Riemannian submersion. 

Adopt the notation from the Appendix. 

Every element $(x,e)=(e_1,\ldots,e_q)\in P_x$, $x\in M$, induces the system of endomorphisms $\bA(x,e)=(A_1(x,e),\ldots,A_q(x,e))$ of $D_x$, where $A_{\alpha}(x,e)$ is the shape operator corresponding to $(x,e)$, i.e.
\begin{equation*}
A_{\alpha}(x,e)(X)=-\left(\nabla_X e_{\alpha}\right)^{\top},\quad X\in D_x.
\end{equation*}  
Let $T(x,e)=(T_u(x,e))_{u\in\mathbb{N}(q)}$ be the generalized Newton transformation associated with $\bA(x,e)$. 

The bundle $\pi:P\to M$ and the vector bundles $TM\to M$, $D\to M$, $D^{\bot}\to M$ induce the pull--back bundles 
\begin{equation*}
E=\pi^{-1}TM,\quad E'=\pi^{-1}D\quad\textrm{and}\quad E''=\pi^{-1}D^{\bot}\quad\textrm{over $P$}, 
\end{equation*}
each with a fiber  $(\pi^{-1}TM)_{(x,e)}=T_xM$, $(\pi^{-1}D)_{(x,e)}=D_x$ and  $(\pi^{-1}D^{\bot})_{(x,e)}=D^{\bot}_x$, respectively. We have
\begin{equation*}
E=E'\oplus E''.
\end{equation*} 
Moreover, the connections $\nabla,\nabla^{\top},\nabla^{\bot}$ of $g$ induce the unique connections $\nabla^E,\nabla^{E'}$ and $\nabla^{E''}$ in $E$, $E'$ and $E''$, respectively (see Appendix).  

\begin{proposition}\label{Prop-EtoEbis}
\begin{enumerate}
\item Let $Y\in\Gamma(E'')$ and $X\in D_x$, $x\in M$. Then, for every $w\in P_x$
\begin{equation*}
\left(\nabla^E_{X^h}Y\right)(w)=\left(\nabla^{E''}_{X^h}Y\right)(w)-A_{Y(w)}(X),
\end{equation*}
where $A_N(X)=-(\nabla_X N)^{\top}$ denotes the shape operator.
\item Let $Y\in\Gamma(E')$ and $X\in D^{\bot}_x$, $x\in M$. Then, for $w\in P_x$
\begin{equation*}
\left(\nabla^E_{X^h}Y\right)(w)=\left(\nabla^{E'}_{X^h}Y\right)(w)
+\left(\nabla_XY\right)^{\bot}(x).
\end{equation*}
\end{enumerate}
\end{proposition}
\begin{proof}
Since $\nabla^E$ and $\nabla^{E''}$ are connections, the operator 
\begin{equation*}
S(X,Y)=\left(\nabla^E_{X^h}Y\right)(w)-\left(\nabla^{E''}_{X^h}Y\right)(w)
\end{equation*}
is tensorial, hence does not depend on the extension of a vector $Y(w)$ to a section of $E$. Thus, we may assume $Y=Y_0\circ\pi$, where $Y_0\in\Gamma(TM)$. Then
\begin{equation*}
S(X,Y)=\left(\nabla_XY_0\right)(x)-\left(\nabla^{\bot}_XY_0\right)(x)=\left(\nabla_XY_0\right)^{\top}(x)
=-A_{Y_0(x)}(X)=-A_{Y(w)}(X),
\end{equation*}
which completes the proof of (1). The proof of (2) is similar.
\end{proof}

Let ${\rm End}(E')$ denotes the bundle of endomorphisms of $E'$, i.e. the fiber ${\rm End}(E')_x$ over $x\in M$ is ${\rm End}(D_x)$. The connection $\nabla^{E'}$ defines a connection in ${\rm End}(E')$.

Each $T_u$ belongs to $\Gamma({\rm End}(E'))$ and $\sigma_u$ is a smooth function on $P$. By the definition of generalized Newton transformation we conclude that
\begin{equation}\label{Eq-diff-sigma}
Z(\sigma_u)=\sum_{\alpha}\tr\left(\left(\nabla^{E'}_ZA\right)_{\alpha}\cdot T_{\alpha_\flat(u)}\right),\quad Z\in\Gamma(P),
\end{equation}
where
\begin{equation*}
\left(\nabla^{E'}_ZA\right)_{\alpha}X=\nabla^{E'}_Z(A_{\alpha}(X))-A_{\alpha}(\nabla^{E'}_ZX)-A_{\nabla^{E''}_Ze_{\alpha}}(X),\quad X\in E'.
\end{equation*}

Applying the notation from Appendix, we have
\begin{equation}\label{eq:sigma_u}
\widehat{\sigma_u}(x)=\int_{P_x}\sigma_u(x,e)\, de=\int_G\sigma_u(x,e_0a)\, da,
\end{equation}
where $(x,e_0)$ is a fixed element of $P_x$. We call $\widehat{\sigma_u}$'s {\it extrinsic curvatures} of a distribution $D$. Moreover, we define {\it total extrinsic curvatures}
\begin{equation}\label{eq:totalsigma_u}
\sigma_u^M=\int_M \widehat{\sigma_u}(x)\, dx.
\end{equation}
Since the projection $\pi$ in the bundle $P$ is a Riemannian submersion, then by Fubini theorem \eqref{Fubini}
\begin{equation*}
\sigma_u^M=\int_P \sigma_u(x,e)\, d(x,e).
\end{equation*} 

\begin{rmk}\label{rmk:sigma0}
Notice that some of total extrinsic curvatures are equal zero. For indices $1\leq\alpha_1<\ldots<\alpha_k\leq q$, let $F_{\alpha_1,\ldots,\alpha_k}$ be a transformation which maps vector $e_{\alpha_i}$ of a basis $(x,e)$ to $-e_{\alpha_i}$, $i=1,\ldots,k$. Then, by the use of the characteristic polynomial of generalized Newton transformation, we get
\begin{equation*}
\sigma_u(F_{\alpha_1,\ldots,\alpha_k}(x,e))=(-1)^{u_{\alpha_1}+\ldots+u_{\alpha_k}}\sigma_u(x,e).
\end{equation*}
Hence, if $u_1+\ldots+u_k$ is odd and $F_{\alpha_1,\ldots,\alpha_k}$ maps $P$ to $P$, then $\widehat{\sigma_u}=0$, so $\sigma^M_u=0$. Thus, the following two conditions hold:
\begin{enumerate}
\item if $G=O(q)$ and at least one of indices $u_1,\ldots,u_q$ is odd, then $\sigma^M_u=0$,
\item if $G=SO(q)$ and there is one index odd and one even, then $\sigma^M_u=0$.
\end{enumerate}
\end{rmk}

Define the section $Y_u\in\Gamma(E)$, $u\in\mathbb{N}(q)$ as follows
\begin{equation}\label{Eq-Y_u}
Y_u(x,e)=\sum_{\alpha,\beta}T_{\beta_\flat\alpha_\flat(u)}(x,e)(\nabla_{e_{\alpha}}e_{\beta})^{\top}
+\sum_{\alpha}\sigma_{\alpha_\flat(u)}(x,e)e_{\alpha}.
\end{equation}
Observe that the first component of $Y_u$ is a section of $E'$, whereas the second component is a section of $E''$. The section $Y_u$ and the vector field $\widehat{Y_u}\in\Gamma(TM)$ obtained from $Y_u$ by integration on the fibers of $P$ play a fundamental role in our considerations. 

We will start by proving necessary technical results. The divergence of $Y_u$ and consequences of obtained formula are contained in subsequent sections. 
 
Let $R_{X,Y}:D\to D$, $X,Y\in\Gamma(TM)$ be an endomorphism given by
\begin{equation*}
R_{X,Y}Z=(R(Z,X)Y)^{\top},\quad Z\in D,
\end{equation*}
where $R$ denotes the curvature tensor of $\nabla$. For fixed indices $\alpha,\beta$ we define a section $R_{\alpha,\beta}$ of ${\rm End}(E')$ by
\begin{equation*}
R_{\alpha,\beta}(x,e)=R_{e_{\alpha},e_{\beta}}:D_x\to D_x.
\end{equation*}
Similarly, we define $A_{\alpha}\in \Gamma({\rm End}(E'))$ by
\begin{equation*}
A_{\alpha}(x,e)=A_{e_{\alpha}}:D_x\to D_x.
\end{equation*}

It is also worth to notice, that $e_{\alpha}$ is identified with a section $e_{\alpha}\in \Gamma(E'')$, which assigns to a basis $(x,e)$ its $\alpha$--th component. Hence, a local orthonormal basis $e=(e_1,\ldots,e_q)$ is considered as a local section of $P$.

\begin{lemma}\label{lem-loc}
Fix $x\in M$. Let $e=(e_1,\ldots,e_q)$ be a local orthonormal frame field in the neighborhood of $x$ such that $(\nabla^{\bot}e_{\alpha})(x)=0$ for all $\alpha$. Extend $e$ to a local orthonormal basis $(f_1,\ldots,f_p,e_1,\ldots,e_q)$ of $TM$ such that $(\nabla^{\top}f_i)(x)=0$ for all $i$. Then, at $x$ we have
\begin{align*}
e_{\alpha}((A_{\beta})_{ij}\circ e) &=(A_{\alpha}A_{\beta})_{ij}\circ e+(R_{\alpha,\beta})_{ij}\circ e-
g(\nabla_{f_i}(\nabla_{e_{\alpha}}e_{\beta})^{\top},f_j)\\
&+\sum_{\gamma}g((\nabla_{e_{\alpha}}e_{\gamma})^{\top},f_i)g(f_j,(\nabla_{e_{\gamma}}e_{\beta})^{\top}).
\end{align*}
where $(\,)_{ij}$ denotes the $(i,j)$--th component with respect to the basis $(f_i)$.
\end{lemma} 
\begin{proof}
Differentiating $g(e_{\beta},f_j)=0$ twice, we have
\begin{equation*}
0 =g(\nabla_{f_i}\nabla_{e_{\alpha}}e_{\beta},f_j)+g(\nabla_{e_{\alpha}}e_{\beta},\nabla_{f_i}f_j)
+g(\nabla_{f_i}e_{\beta},\nabla_{e_{\alpha}}f_j)+g(e_{\beta},\nabla_{f_i}\nabla_{e_{\alpha}}f_j).
\end{equation*}
Thus, at the point $x$
\begin{equation}\label{lem-loc-e1}
g(\nabla_{f_i}\nabla_{e_{\alpha}}e_{\beta},f_j)=-g(e_{\beta},\nabla_{f_i}\nabla_{e_{\alpha}}f_j).
\end{equation}
Moreover 
\begin{align*}
(A_{\beta})_{ij}\circ e &=-g(\nabla_{f_i}e_{\beta},f_j),\\
(A_{\alpha}A_{\beta})\circ e &=\sum_kg(\nabla_{f_i}e_{\alpha},f_k)g(\nabla_{f_k}e_{\beta},f_j),\\
(R_{\alpha\beta})_{ij}\circ e &=g(R(f_i,e_{\alpha})e_{\beta},f_j).
\end{align*}
Put
\begin{equation*}
L=e_{\alpha}((A_{\beta})_{ij}\circ e)-(A_{\alpha}A_{\beta})_{ij}\circ e-(R_{\alpha,\beta})_{ij}\circ e.
\end{equation*}
Then, at the point $x$
\begin{align*}
L&=e_{\alpha}g(e_{\beta},\nabla_{f_i}f_j)-\sum_kg(\nabla_{f_i}e_{\alpha},f_k)g(\nabla_{f_k}e_{\beta},f_j)
+g(R(f_i,e_{\alpha})f_j,e_{\beta})\\
&=-\sum_kg(\nabla_{f_i}e_{\alpha},f_k)g(\nabla_{f_k}e_{\beta},f_j)+g(\nabla_{f_i}\nabla_{e_{\alpha}}f_j,e_{\beta})
-g(\nabla_{[f_i,e_{\alpha}]}f_j,e_{\beta})\\
&=-\sum_kg(\nabla_{f_i}e_{\alpha},f_k)g(\nabla_{f_k}e_{\beta},f_j)+g(\nabla_{f_i}\nabla_{e_{\alpha}}f_j,e_{\beta})\\
&-\sum_kg([f_i,e_{\alpha}],f_k)g(\nabla_{f_k}f_j,e_{\beta})
-\sum_{\gamma}g([f_i,e_{\alpha}],e_{\gamma})g(\nabla_{e_{\gamma}}f_j,e_{\beta})\\
&=g(\nabla_{f_i}\nabla_{e_{\alpha}}f_j,e_{\beta})
+\sum_{\gamma}g(\nabla_{e_{\alpha}}f_i,e_{\gamma})g(\nabla_{e_{\gamma}}f_j,e_{\beta}).
\end{align*}
Thus by \eqref{lem-loc-e1}
\begin{align*}
L &=-g(\nabla_{f_i}\nabla_{e_{\alpha}}e_{\beta},f_j)
+\sum_{\gamma}g(\nabla_{e_{\alpha}}f_i,e_{\gamma})g(\nabla_{e_{\gamma}}f_j,e_{\beta})\\
&=-g(\nabla_{f_i}\nabla^{\top}_{e_{\alpha}}e_{\beta},f_j)
+\sum_{\gamma}g(f_i,\nabla^{\top}_{e_{\alpha}}e_{\gamma})g(f_j,\nabla^{\top}_{e_{\gamma}}e_{\beta}),
\end{align*}
which completes the proof.
\end{proof}

\section{Integral formulas}

Adopt the notation from the previous section and from Appendix. The main result of this section is the integral formula for the total extrinsic curvatures. This formula is derived by computation of a divergence of vector field $\widehat{Y_u}$ corresponding to the section $Y_u$ defined by \eqref{Eq-Y_u}.

First, we recall that the divergence of a $(1,m)$--tensor $S$ on $M$ is a $(1,m-1)$--tensor ${\rm div}S$ of the form
\begin{equation*}
({\rm div}S)(X_1,\ldots,X_{m-1})=\sum_{\mu=1}^n(\nabla_{e_{\mu}}S)(e_{\mu},X_1,\ldots,X_{m-1}),
\end{equation*} 
where $(e_{\mu})$ is any local orthonormal basis of $TM$. Considering only the basis adapted to a distribution $D$, we define analogously the divergence ${\rm div}^{\top}S$ with respect to $D$. 

For a linear map $S:D_x\to D_x$ let $S^{\ast}:D_x\to D_x$ denotes the operator adjoint to $S$, i.e., $g(SX,Y)=g(X,S^{\ast}Y)$ for $X,Y\in D_x$.
Let $S:D\to D$ be a tensor field on $M$. Then one can prove the following relations
\begin{align}
&\nabla^{\top}_XS^{\ast}=\left(\nabla^{\top}_X S\right)^{\ast},\label{div-ast1} \\
&g({\rm div}^{\top}S^{\ast},Y)={\rm div}^{\top}(SY)-\sum_i g(S^{\ast}f_i,\nabla^{\top}_{f_i}Y),\label{div-ast2}
\end{align}
where $(f_i)$ is local orthonormal basis of $D$, $X\in TM$, $Y\in D$.

We will use the following notation: if $X\in\Gamma(TM)$ and $Z\in \Gamma(E')$ then $g(X,Z)$ denotes the function on $P$ of the form $g(X,Z)(x,e)=g(X_x,Z(x,e))$, where $(x,e)\in P$.

\begin{lemma}\label{Lem-divY_u}
The divergence of $Y_u$ is given by the formula
\begin{multline*}
{\rm div}_EY_u =-|u|\sigma_u+\sum_{\alpha,\beta}\Big[\tr(R_{\alpha,\beta}T_{\beta_\flat\alpha_\flat(u)})+g({\rm div}_{E'}T^{\ast}_{\beta_\flat\alpha_\flat(u)},(\nabla_{e_{\alpha}}e_{\beta})^{\top})\\
-g(H_{D^{\bot}},T_{\beta_\flat\alpha_\flat(u)}(\nabla_{e_{\alpha}}e_{\beta})^{\top})
+\sum_{\gamma}g((\nabla_{e_{\alpha}}e_{\gamma})^{\top},T_{\beta_\flat\alpha_\flat(u)}(\nabla_{e_{\gamma}}e_{\beta})^{\top})
\Big],
\end{multline*}
where $H_{D^{\bot}}$ denotes the mean curvature vector of the distribution $D^{\bot}$.
\end{lemma}
\begin{proof}
By \eqref{appx-divfunc}, \eqref{Eq-diff-sigma} and Proposition \ref{Prop-EtoEbis} we have
\begin{align*}
{\rm div}_E\left(\sum_{\alpha}\sigma_{\alpha_\flat(u)}e_{\alpha}\right)
&=\sum_{\alpha}e_{\alpha}^h(\sigma_{\alpha_\flat(u)})+\sum_{\alpha}{\rm div}_E(e_{\alpha} )\cdot\sigma_{\alpha_\flat(u)}\\
&=\sum_{\alpha,\beta}\tr((\nabla^{E'}_{e_{\alpha}^h}A)_{\beta}\cdot T_{\beta_\flat\alpha_\flat(u)})+\sum_{\alpha,i}g(\nabla^E_{f_i^h}e_{\alpha},f_i\circ\pi)\sigma_{\alpha_\flat(u)}\\
&+\sum_{\alpha,\beta}g(\nabla^E_{e_{\beta}^h}e_{\alpha},e_{\beta}\circ\pi)\sigma_{\alpha_\flat(u)}\\
&=\sum_{\alpha,\beta}\tr((\nabla^{E'}_{e_{\alpha}^h}A)_{\beta}\cdot T_{\beta_\flat\alpha_\flat(u)})
-\sum_{\alpha}\tr(A_{\alpha})\sigma_{\alpha_\flat(u)}\\
&+\sum_{\alpha,\beta}g(\nabla^{E''}_{e_{\beta}^h}e_{\alpha},e_{\beta}\circ\pi)\sigma_{\alpha_\flat(u)},
\end{align*}
where $(f_i)$ is an orthonormal basis of $D$ and $(e_{\beta})$ an orthonormal basis of $D{^\bot}$. Moreover, again by Proposition \ref{Prop-EtoEbis},
\begin{align*}
{\rm div}_E\left( \sum_{\alpha,\beta}T_{\beta_\flat\alpha_\flat(u)}(\nabla_{e_{\alpha}}e_{\beta})^{\top}\right) &=\sum_{\alpha,\beta,i}g(\nabla^{E'}_{f_i^h}T_{\beta_\flat\alpha_\flat(u)}(\nabla_{e_{\alpha}}e_{\beta})^{\top},f_i\circ\pi)\\
&+\sum_{\alpha,\beta,\gamma}g(\nabla_{e_{\gamma}}T_{\beta_\flat\alpha_\flat(u)}(\nabla_{e_{\alpha}}e_{\beta})^{\top},e_{\gamma}\circ\pi).
\end{align*}
Fix $x\in M$. Let $e=(e_1,\ldots,e_q)$ be a local orthonormal frame field in the neighborhood of $x$ such that $(\nabla^{\bot}e_{\alpha})(x)=0$. Extend $e$ to a local orthonormal basis $(f_1,\ldots,f_p,e_1,\ldots,e_q)$ of $TM$ such that $(\nabla^{\top}f_i)(x)=0$. Then, by above formulas, \eqref{appx-pr2} and Lemma \ref{lem-loc} we have at point $(x,e)$ (we omit composition with basis e and projection $\pi$)
\begin{align*}
{\rm div}_EY_u&=\sum_{\alpha,\beta}\tr((\nabla^{\top}_{e_{\alpha}}A_{\beta})\cdot T_{\beta_\flat\alpha_\flat(u)})
-\sum_{\alpha}\tr(A_{\alpha})\sigma_{\alpha_\flat(u)}\\
&+\sum_{\alpha,\beta,i}g(\nabla^{\top}_{e_i}T_{\beta_\flat\alpha_\flat(u)}(\nabla_{e_{\alpha}}e_{\beta})^{\top},e_i)
+\sum_{\alpha,\beta,\gamma}g(\nabla_{e_{\gamma}}T_{\beta_\flat\alpha_\flat(u)}(\nabla_{e_{\alpha}}e_{\beta})^{\top},e_{\gamma})\\
&=-\sum_{\alpha}\tr(A_{\alpha})\sigma_{\alpha_\flat(u)}
+\sum_{\alpha,\beta,i,j}(e_{\alpha}(A_{\beta})_{ij}) (T_{\beta_\flat\alpha_\flat(u)})_{ji}\\
&+\sum_{\alpha,\beta,i}g(\nabla^{\top}_{e_i}T_{\beta_\flat\alpha_\flat(u)}(\nabla_{e_{\alpha}}e_{\beta})^{\top},e_i)
-\sum_{\alpha,\beta,\gamma}g(T_{\beta_\flat\alpha_\flat(u)}(\nabla_{e_{\alpha}}e_{\beta})^{\top},\nabla_{e_{\gamma}}e_{\gamma})\\
&=-\sum_{\alpha}\tr(A_{\alpha})\sigma_{\alpha_\flat(u)}
+\sum_{\alpha,\beta}\tr(A_{\alpha}A_{\beta}T_{\beta_\flat\alpha_\flat(u)})
+\sum_{\alpha,\beta}\tr(R_{\alpha,\beta}T_{\beta_\flat\alpha_\flat(u)})\\
&-
\sum_{\alpha,\beta,i,j}g(\nabla_{f_i}(\nabla_{e_{\alpha}}e_{\beta})^{\top},f_j)(T_{\beta_\flat\alpha_\flat(u)})_{ji}\\
&+\sum_{\alpha,\beta,\gamma,i,j}g((\nabla_{e_{\alpha}}e_{\gamma})^{\top},f_i)g(f_j,(\nabla_{e_{\gamma}}e_{\beta})^{\top})(T_{\beta_\flat\alpha_\flat(u)})_{ji}\\
&+\sum_{\alpha,\beta}{\rm div}^{\top}(T_{\beta_\flat\alpha_\flat(u)}(\nabla_{e_{\alpha}}e_{\beta})^{\top})
-g(H_{D^{\bot}},T_{\beta_\flat\alpha_\flat(u)}(\nabla_{e_{\alpha}}e_{\beta})^{\top}).
\end{align*}
By Proposition \ref{Prop-GNT} (GN3) we obtain
\begin{align*}
{\rm div}_EY_u &=-|u|\sigma_u+\sum_{\alpha,\beta}\tr(R_{\alpha,\beta}T_{\beta_\flat\alpha_\flat(u)})
-\sum_{\alpha,\beta,i,j}g(\nabla_{f_i}(\nabla_{e_{\alpha}}e_{\beta})^{\top},(T_{\beta_\flat\alpha_\flat(u)})_{ji}f_j)\\
&+\sum_{\alpha,\beta,\gamma,i}g((\nabla_{e_{\alpha}}e_{\gamma})^{\top},f_i)g((T_{\beta_\flat\alpha_\flat(u)})_{ji}f_j,\nabla^{\top}_{e_{\gamma}}e_{\beta})\\
&+\sum_{\alpha,\beta}{\rm div}^{\top}(T_{\beta_\flat\alpha_\flat(u)}(\nabla_{e_{\alpha}}e_{\beta})^{\top})
-g(H_{D^{\bot}},T_{\beta_\flat\alpha_\flat(u)}(\nabla_{e_{\alpha}}e_{\beta})^{\top})\\
&=-|u|\sigma_u+\sum_{\alpha,\beta}\tr(R_{\alpha,\beta}T_{\beta_\flat\alpha_\flat(u)})
-\sum_{\alpha,\beta,i}g(\nabla_{f_i}(\nabla_{e_{\alpha}}e_{\beta})^{\top},T^{\ast}_{\beta_\flat\alpha_\flat(u)}f_i)\\
&+\sum_{\alpha,\beta,\gamma,i}g((\nabla_{e_{\alpha}}e_{\gamma})^{\top},f_i)g((T^{\ast}_{\beta_\flat\alpha_\flat(u)}f_i,\nabla^{\top}_{e_{\gamma}}e_{\beta})\\
&+\sum_{\alpha,\beta}{\rm div}^{\top}(T_{\beta_\flat\alpha_\flat(u)}(\nabla_{e_{\alpha}}e_{\beta})^{\top})
-g(H_{D^{\bot}},T_{\beta_\flat\alpha_\flat(u)}(\nabla_{e_{\alpha}}e_{\beta})^{\top}).
\end{align*}
Putting $Y=(\nabla_{e_{\alpha}}e_{\beta})^{\top}$ and $S=T_{\beta_\flat\alpha_\flat(u)}$ in \eqref{div-ast2} we get
\begin{equation*}
g({\rm div}^{\top}T_{\beta_\flat\alpha_\flat(u)}^{\ast},(\nabla_{e_{\alpha}}e_{\beta})^{\top})
={\rm div}^{\top}(T_{\beta_\flat\alpha_\flat(u)}(\nabla_{e_{\alpha}}e_{\beta})^{\top})-\sum_i g(T_{\beta_\flat\alpha_\flat(u)}^{\ast}f_i,\nabla^{\top}_{f_i}(\nabla_{e_{\alpha}}e_{\beta})^{\top}).
\end{equation*}
Hence
\begin{align*}
{\rm div}_EY_u &=-|u|\sigma_u+\sum_{\alpha,\beta}\tr(R_{\alpha,\beta}T_{\beta_\flat\alpha_\flat(u)})
+\sum_{\alpha,\beta}g({\rm div}^{\top}T_{\beta_\flat\alpha_\flat(u)}^{\ast},(\nabla_{e_{\alpha}}e_{\beta})^{\top})\\
&+\sum_{\alpha,\beta,\gamma}g((T^{\ast}_{\beta_\flat\alpha_\flat(u)}(\nabla_{e_{\alpha}}e_{\gamma})^{\top},
(\nabla_{e_{\gamma}}e_{\beta})^{\top})-g(H_{D^{\bot}},T_{\beta_\flat\alpha_\flat(u)}(\nabla_{e_{\alpha}}e_{\beta})^{\top}).
\end{align*}
Moreover, by Proposition \ref{appx-pr6} ${\rm div}^{\top}T_{\beta_\flat\alpha_\flat(u)}^{\ast}={\rm div}_{E'}T_{\beta_\flat\alpha_\flat(u)}^{\ast}$, which completes the proof.
\end{proof}

\begin{theorem}\label{Thm-main1}
Assume $M$ is closed. Then, for any $u\in\mathbb{N}(q)$, the total extrinsic curvature $\sigma^M_u$ satisfies
\begin{equation}\label{main-eq1}
\begin{split}
|u|\sigma^M_u &=\sum_{\alpha,\beta}\int_P \Big(\tr(R_{\alpha,\beta}T_{\beta_\flat\alpha_\flat(u)})
+g({\rm div}_{E'}T_{\beta_\flat\alpha_\flat(u)}^{\ast},(\nabla_{e_{\alpha}}e_{\beta})^{\top})\\
&-g(H_{D^{\bot}},T_{\beta_\flat\alpha_\flat(u)}(\nabla_{e_{\alpha}}e_{\beta})^{\top})
+\sum_{\gamma}g((T^{\ast}_{\beta_\flat\alpha_\flat(u)}
(\nabla_{e_{\alpha}}e_{\gamma})^{\top},(\nabla_{e_{\gamma}}e_{\beta})^{\top})\Big),
\end{split}
\end{equation}
where $H_{D^{\bot}}$ denotes the mean curvature vector of distribution $D^{\bot}$.
\end{theorem}
\begin{proof}
Follows immediately by Lemma \ref{Lem-divY_u} and Proposition \ref{appx-pr5}.
\end{proof}

By Theorem \ref{Thm-main1}, we have in particular
\begin{equation*}
\sigma^M_{\alpha^\sharp(0,\ldots,0)}=0
\end{equation*}
and
\begin{multline}\label{eq:sigma2}
2\sigma^M_{\alpha^\sharp\beta^\sharp(0,\ldots,0)}=\int_P \Big(\left({\rm Ric}_D\right)_{\alpha,\beta}
-g(H_{D^{\bot}},(\nabla_{e_{\alpha}}e_{\beta})^{\top})\\
+\sum_{\gamma}g((\nabla_{e_{\alpha}}e_{\gamma})^{\top},(\nabla_{e_{\gamma}}e_{\beta})^{\top})\Big),
\end{multline}
where $({\rm Ric}_D)_{\alpha,\beta}={\rm Ric}_D(e_{\alpha},e_{\beta})$ and ${\rm Ric}_D$ is the Ricci curvature operator in the direction of $D$, i.e.,
\begin{equation*}
{\rm Ric}_D(X,Y)=\sum_i g(R(f_i,X)Y,f_i),
\end{equation*}
where $(f_i)$ is an orthonormal basis of $D$.

Notice that in Theorem \ref{Thm-main1} total extrinsic curvature $\sigma^M_u$ is expressed by generalized Newton transformations $T_{\alpha_\flat\beta_\flat(u)}$ of lower order and by its divergence. On the other hand, we have a recurrence formula \eqref{Eq-T_u} for generalized Newton transformation. We derive the recurrence formula for the divergence of $T^{\ast}_u$ and finally the explicit formula for ${\rm div}_{E'}T^{\ast}_u$.

From Codazzi formula, if $(\nabla N)^{\bot}=0$, where $N\in \Gamma(D^{\bot})$, it follows that
\begin{equation}\label{Codazzi}
(\nabla^{\top}_XA_N)Y-(\nabla^{\top}_YA_N)X=-(R(X,Y)N)^{\top}+(\nabla_{[X,Y]^{\bot}}N)^{\top}
\end{equation}
for $X,Y\in D$.

\begin{proposition}\label{Prop-divTast}
The divergence ${\rm div}_{E'}T^{\ast}_u$, $u\in\mathbb{N}(q)$, satisfies the recurrence formula
\begin{align*}
{\rm div}_{E'}T^{\ast}_u &=\sum_{\alpha}\left( \tr_D(R(e_{\alpha},T^{\ast}_{\alpha_\flat(u)}\cdot)\cdot)^{\top}-A^{\ast}_{\alpha}({\rm div}_{E'}T^{\ast}_{\alpha_\flat(u)}) \right)\\
&-\sum_{\alpha,\beta}(A_{\beta}-A^{\ast}_{\beta})T_{\alpha_\flat(u)}(\nabla_{e_{\beta}}e_{\alpha})^{\top}.
\end{align*}
\end{proposition}
\begin{proof}
By \eqref{Eq-T_u} we have $T^{\ast}_u=\sigma_u\cdot{\rm id}_D-\sum_{\alpha}A^{\ast}_{\alpha}T^{\ast}_{\alpha_\flat(u)}$. Since ${\rm div}_{E'}({\rm id}_D)=0$, by \eqref{divE1} and \eqref{divE2} we get
\begin{equation}\label{Eq-divt1}
{\rm div}_{E'}T^{\ast}_u=\pi_{\ast}({\rm grad}_{D^h}\sigma_u)
-\sum_{\alpha,i}(\nabla^{E'}_{f_i^h}A_{\alpha}^{\ast})(T_{\alpha_\flat(u)}f_i)
-\sum_{\alpha}A^{\ast}_{\alpha}({\rm div}_{E'}T^{\ast}_{\alpha_\flat(u)}),
\end{equation}
where ${\rm grad}_{D^h}(\sigma_u)$ denotes the $D^h$--component of gradient of $\sigma_u$ and $(f_i)$ is a local orthonormal basis of $D$. Fix $x\in M$. Let $e=(e_1,\ldots,e_q)$ be a local section of $P$ such that $(\nabla^{\bot}e_{\alpha})(x)=0$. Then, for any vector $X\in D_x$
\begin{equation*}
g(\pi_{\ast}({\rm grad}_{D^h}\sigma_u),X)=X(\sigma_u\circ e)
\end{equation*}
and, by Proposition \ref{appx-pr6}, \eqref{Codazzi} and \eqref{div-ast1},
\begin{align*}
g(\sum_{\alpha,i}(\nabla^{E'}_{f_i^h}A_{\alpha}^{\ast})(T^{\ast}_{\alpha_\flat(u)}f_i),X) &=
\sum_{\alpha,i}g((\nabla^{\top}_{f_i}A_{\alpha})^{\ast}T^{\ast}_{\alpha_\flat(u)}f_i,X)\\
&=\sum_{\alpha,i}g(T^{\ast}_{\alpha_\flat(u)}f_i,(\nabla^{\top}_{f_i}A_{\alpha})X)\\
&=\sum_{\alpha,i}g(T^{\ast}_{\alpha_\flat(u)}f_i,(\nabla^{\top}_XA_{\alpha})f_i)\\
&+\sum_{\alpha,i}g(T^{\ast}_{\alpha_\flat(u)}f_i,-(R(f_i,X)e_{\alpha})^{\top}+(\nabla_{[f_i,X]^{\bot}}e_{\alpha})^{\top})\\
&=\sum_{\alpha}\tr((\nabla^{\top}_XA_{\alpha}\cdot T_{\alpha_\flat(u)})
-\sum_{\alpha,i}g(R(e_{\alpha},T^{\ast}_{\alpha_\flat(u)}f_i)f_i,X)\\
&+\sum_{\alpha,i}g(T^{\ast}_{\alpha_\flat(u)}f_i,(\nabla_{[f_i,X]^{\bot}}e_{\alpha})^{\top}).
\end{align*}
Hence, by \eqref{Eq-diff-sigma}
\begin{multline}\label{Eq-divt2}
\pi_{\ast}({\rm grad}_{D^h}\sigma_u)
-\sum_{\alpha,i}(\nabla^{E'}_{f_i^h}A_{\alpha}^{\ast})(T_{\alpha_\flat(u)}f_i)\\
=\sum_{\alpha,i}g(R(e_{\alpha},T^{\ast}_{\alpha_\flat(u)}f_i)f_i,X)
-\sum_{\alpha,i}g(T^{\ast}_{\alpha_\flat(u)}f_i,(\nabla_{[f_i,X]^{\bot}}e_{\alpha})^{\top}).
\end{multline}
Moreover
\begin{align}
\sum_ig(T^{\ast}_{\alpha_\flat(u)}f_i,(\nabla_{[f_i,X]^{\bot}}e_{\alpha})^{\top}) &=
\sum_{\beta,i}g(T^{\ast}_{\alpha_\flat(u)}f_i,(\nabla_{e_{\beta}}e_{\alpha})^{\top})g(e_{\beta},\nabla_{f_i}X-\nabla_Xf_i)\notag\\
&=\sum_{\beta,i}g(f_i,T_{\alpha_\flat(u)}(\nabla_{e_{\beta}}e_{\alpha})^{\top})g(X,(A_{\beta}-A_{\beta}^{\ast})f_i)\label{Eq-divt3}\\
&=\sum_{\beta}g(X,(A_{\beta}-A_{\beta}^{\ast})T_{\alpha_\flat(u)}(\nabla_{e_{\beta}}e_{\alpha})^{\top}).\notag
\end{align}
By \eqref{Eq-divt1}--\eqref{Eq-divt3} proposition follows.
\end{proof}

\begin{cor}\label{Cor-divTu}
The divergence ${\rm div}_{E'}T^{\ast}_u$ of $T^{\ast}_u$ can be expressed as follows
\begin{align*}
{\rm div}_{E'}T^{\ast}_u &=\sum_{1\leq s\leq |u|}\sum_{\alpha_1,\ldots,\alpha_s}
(-1)^{s-1}A^{\ast}_{\alpha_1}\ldots A^{\ast}_{\alpha_{s-1}}\left(\tr_D R(e_{\alpha_s},T^{\ast}_{(\alpha_s)_\flat\ldots(\alpha_1)_\flat(u)}\cdot)\cdot \right)^{\top}\\
&-\sum_{1\leq s\leq |u|}\sum_{\alpha_1,\ldots,\alpha_j,\beta}(-1)^{s-1}A^{\ast}_{\alpha_1}\ldots A^{\ast}_{\alpha_{s-1}}(A_{\beta}-A^{\ast}_{\beta})T_{(\alpha_s)_\flat\ldots(\alpha_1)_\flat(u)}(\nabla_{e_{\beta}}e_{\alpha_s})^{\top}).
\end{align*}
\end{cor}
\begin{proof}
Follows by induction with respect to $|u|$ and the recurrence formula for ${\rm div}_{E'}T^{\ast}_u$ in Proposition \ref{Prop-divTast}.
\end{proof}

\section{Consequences and special cases}

In this section we show some applications of the formula for the total extrinsic curvatures. We derive the generalization of the Walczak formula \cite{Wal}. Moreover, we compute total extrinsic curvatures for foliations with integrable and totally geodesic normal bundle. We conclude by showing that in the case of one shape operator the results agree with the ones obtained in \cite{Rov} and in the codimension one with the formulas in \cite{AW1}.

We adopt notation from the previous section.

\subsection{Generalization of Walczak formula}

Let $D$ be a $p$--dimensional distribution on a $(p+q)$--dimensional Riemannian manifold $(M,g)$. By Proposition \ref{Prop-GNT}(GN1) and (GN3) we have
\begin{equation*}
(|u|-1)|u|\sigma_u=\sum_{\alpha,\beta}\left( \tr(A_{\alpha})\tr(A_{\beta}T_{\beta_\flat\alpha_\flat(u)})-\tr(A_{\alpha}A_{\beta}T_{\beta_\flat\alpha_\flat(u)}) \right).
\end{equation*} 
Together with \eqref{main-eq1} we obtain the following result.

\begin{cor}\label{Cor-Walfor}
Assume $M$ is closed. For any multi--index $u\in\mathbb{N}(q)$, $|u|>1$, we have the following integral formula
\begin{align}
0 =\sum_{\alpha,\beta}\int_P &\Big(\tr(R_{\alpha,\beta}T_{\beta_\flat\alpha_\flat(u)})
+g({\rm div}_{E'}T_{\beta_\flat\alpha_\flat(u)}^{\ast},(\nabla_{e_{\alpha}}e_{\beta})^{\top}) \notag \\
&-g(H_{D^{\bot}},T_{\beta_\flat\alpha_\flat(u)}(\nabla_{e_{\alpha}}e_{\beta})^{\top})
+\sum_{\gamma}g(T^{\ast}_{\beta_\flat\alpha_\flat(u)}
(\nabla_{e_{\alpha}}e_{\gamma})^{\top},(\nabla_{e_{\gamma}}e_{\beta})^{\top})\label{eq:Walfor}\\
&-\frac{1}{|u|-1}\left( \tr(A_{\alpha})\tr(A_{\beta}T_{\beta_\flat\alpha_\flat(u)})-\tr(A_{\alpha}A_{\beta}T_{\beta_\flat\alpha_\flat(u)}) \right)\Big)\notag
\end{align}
\end{cor}

Formula \eqref{eq:Walfor} is a generalization of Walczak formula \cite{Wal}. Indeed, to state and prove the Walczak formula let us now introduce some necessary definitions. The {\it second fundamental form} $B_D$ and integrability tensor $T_D$ of a distribution $D$ are bilinear forms symmetric and skew--symmetric, respectively, given by the formulas
\begin{align*}
B_D(X,Y) &=\left( \nabla_XY+\nabla_YX \right)^{\bot},\\
T_D(X,Y) &=\left( \nabla_XY-\nabla_YX \right)^{\bot}=[X,Y]^{\bot},\quad X,Y\in D.
\end{align*} 
Moreover, the {\it mixed scalar curvature} in the direction of $D$ and $D^{\bot}$ is defined as follows
\begin{equation*}
K(D,D^{\bot})=\sum_{i,\alpha}K(f_i,e_{\alpha})=\sum_{i,\alpha}g(R(f_i,e_{\alpha})e_{\alpha},f_i)=
\sum_{\alpha}\tr R_{\alpha,\alpha},
\end{equation*}
where $(f_i)$ and $(e_{\alpha})$ are orthonormal basis of $D$ and $D^{\bot}$, respectively. Then, Walczak formula is the following
\begin{equation}\label{eq:Walczakformula}
\int_M\Big( K(D,D^{\bot})-|H_D|^2-|H_{D^{\bot}}|^2+|B_D|^2+|B_{D^{\bot}}|^2-|T_D|^2-|T_{D^{\bot}}|^2 \Big)=0.
\end{equation} 

Any multi--index $u$ of length $2$ with even coordinates is of the form $u=\alpha^\sharp\alpha^\sharp(0,\ldots,0)$. For such an multi--index formula \eqref{eq:Walfor} reduces to (see also \eqref{eq:sigma2})
\begin{equation}\label{eq:wf1}
\begin{split}
0=\sum_{\alpha}\int_P &\Big( \tr R_{\alpha,\alpha}-g(H_{D{^\bot}},(\nabla_{e_{\alpha}}e_{\alpha})^{\bot})+
\sum_{\gamma}g((\nabla_{e_{\alpha}}e_{\gamma})^{\bot},(\nabla_{e_{\gamma}}e_{\alpha})^{\bot})\\
&-(\tr A_{\alpha})^2+\tr(A_{\alpha}^2)\Big).
\end{split}
\end{equation}
Moreover, we have
\begin{equation*}
\sum_{\alpha}(\tr A_{\alpha})^2=\sum_{\alpha,i}g(-\nabla_{f_i}e_{\alpha},f_i)^2=\sum_{\alpha}g(e_{\alpha},H_D)^2=|H_D|^2
\end{equation*}
and
\begin{align*}
\sum_{\alpha}\tr(A_{\alpha}^2) &=\sum_{\alpha,i}g(A_{\alpha}(A_{\alpha}f_i,f_i)=
\sum_{\alpha,i}g(A_{\alpha}f_i,A^{\ast}_{\alpha}f_i)\\
&=\sum_{\alpha,i,j}g(A_{\alpha}f_i,f_j)g(A_{\alpha}f_j,f_i)\\
&=\sum_{\alpha,i,j}g(e_{\alpha},B_D(f_i,f_j)+T_D(f_i,f_j))g(e_{\alpha},B_D(f_i,f_j)-T_D(f_i,f_j))\\
&=|B_D|^2-|T_D|^2
\end{align*}
and
\begin{align*}
\sum_{\gamma}g((\nabla_{e_{\alpha}}e_{\gamma})^{\bot},(\nabla_{e_{\gamma}}e_{\alpha})^{\bot})&=
\sum_{\gamma}g(B_{D^{\bot}}(e_{\alpha},e_{\gamma})+T_{D^{\bot}}(e_{\alpha},e_{\gamma}))\\
&\cdot g(B_{D^{\bot}}(e_{\gamma},e_{\alpha})+T_{D^{\bot}}(e_{\gamma},e_{\alpha}))\\
&=|B_{D^{\bot}}|^2-|T_{D^{\bot}}|^2.
\end{align*}
Using above equalities and the fact that all obtained functions are constant on the fibers $P_x$ formula \eqref{eq:wf1} reduces to Walczak formula \eqref{eq:Walczakformula}.

\subsection{Distributions with totally geodesic and integrable normal bundle}

Following \cite{Rei2,CL} we define the $r$--th mean extrinsic curvatures $S_r$ and give the integral formulas for these quantities. Next, we compute total extrinsic curvatures $\sigma_u^M$ in the case of a distribution with integrable and totally geodesic normal bundle on a Riemannian manifold of constant sectional curvature and show that obtained result implies the formula for $S_r$ obtained by Brito and Naveira \cite{BN}.  

Let $D$ be a distribution on $M$. For $r$ even define the $r$--th {\it mean extrinsic curvature} $S_r$ by
\begin{equation*}
S_r=\frac{1}{r!}\sum_{\stackrel{i_1,\ldots,i_r}{ j_1,\ldots,j_r}}\delta^{i_1,\ldots,i_r}_{j_1,\ldots,j_r}\sum_{\alpha_1,\ldots,\alpha_{\frac{r}{2}}}
(A_{\alpha_1})_{i_1j_1}(A_{\alpha_1})_{i_2j_2}\ldots(A_{\alpha_{\frac{r}{2}}})_{i_{r-1}j_{r-1}}
(A_{\alpha_{\frac{r}{2}}})_{i_rj_r},
\end{equation*}
where $\delta^{i_1,\ldots,i_r}_{j_1,\ldots,j_r}$ is a generalized Kronecker symbol (see Section 2) and $A_{ij}$ denotes the coefficients of an endomorphism $A$ with respect to an orthonormal basis. One can show that $S_r$ does not depend on the choice of orthonormal basis, hence is a well defined function on $M$. By Theorem \ref{thm:relTrTu} we get
\begin{equation*}
S_r=\sum_{\stackrel{u\in 2\mathbb{N}(q)}{|u|=r}} \binom{\frac{r}{2}}{\frac{u}{2}}\binom{r}{u}^{-1}\,\sigma_u.
\end{equation*}
Hence, by Theorem \ref{Thm-main1}, we get the integral formula for $S_r$.

\begin{cor}\label{cor:formulaS_r}
The $r$--th mean extrinsic curvature $S_r$ on closed Riamannian manifold satisfies the following integral formula
\begin{equation*}
\begin{split}
rS_r &=\sum_{\stackrel{u\in 2\mathbb{N}(q)}{|u|=r}} \binom{\frac{r}{2}}{\frac{u}{2}}\binom{r}{u}^{-1}\sum_{\alpha,\beta}\int_P \Big(\tr(R_{\alpha,\beta}T_{\beta_\flat\alpha_\flat(u)})
+g({\rm div}_{E'}T_{\beta_\flat\alpha_\flat(u)}^{\ast},(\nabla_{e_{\alpha}}e_{\beta})^{\top})\\
&-g(H_{D^{\bot}},T_{\beta_\flat\alpha_\flat(u)}(\nabla_{e_{\alpha}}e_{\beta})^{\top})
+\sum_{\gamma}g((T^{\ast}_{\beta_\flat\alpha_\flat(u)}
(\nabla_{e_{\alpha}}e_{\gamma})^{\top},(\nabla_{e_{\gamma}}e_{\beta})^{\top})\Big),
\end{split}
\end{equation*}
where $H_{D^{\bot}}$ denotes the mean curvature vector of $D^{\bot}$.
\end{cor}

Let $D$ be a distribution such that the bundle $D^{\bot}$ is totally geodesic and integrable. Then $\left(\nabla_XY\right)^{\top}=0$ for any $X,Y\in D^{\bot}$.

\begin{cor}\label{Cor-Dbottotgeod}
Assume $M$ is closed. Then, for any $u\in\mathbb{N}(q)$, total extrinsic curvature $\sigma^M_u$ of a distribution $D$ with totally geodesic normal bundle is of the form
\begin{equation*}
|u|\sigma^M_u=\sum_{\alpha,\beta}\int_P\tr(R_{\alpha,\beta}T_{\beta_\flat\alpha_\flat(u)}).
\end{equation*}
\end{cor}

Assume additionally that $(M,g)$ is of constant sectional curvature $\kappa$. Then $R_{\alpha,\beta}=\kappa\delta_{\alpha,\beta}$, where $\delta_{\alpha,\beta}$ is the Kronecker symbol. Therefore, by Proposition \ref{Prop-GNT} (GN2) and Corollary \ref{Cor-Dbottotgeod}, we have
\begin{equation}\label{eq:sigmau_totgeodesic}
\begin{split}
|u|\sigma^M_u &=\kappa\sum_{\alpha}\int_P \tr(T_{\alpha_\flat^2(u)})\\
&=\kappa\sum_{\alpha}\int_P (p-|u|+2)\sigma_{\alpha_\flat^2(u)}\\
&=\kappa(p-|u|+2)\sum_{\alpha}\sigma^M_{\alpha_\flat^2(u)}.
\end{split}
\end{equation}
This, together with 
\begin{equation*}
\sigma^M_{(0,\ldots,0)}={\rm vol}(P)\quad\textrm{and}\quad \sigma_{\alpha^\sharp(0,\ldots,0)}^M=0
\end{equation*} 
this gives the recurrence relation for total extrinsic curvatures. 

\begin{cor}\label{cor:totgeodnormal}
Assume $(M,g)$ is closed and of constant sectional curvature $\kappa$. Let $\mathcal{F}$ be a foliation on $M$ with totally geodesic and integrable normal bundle $\mathcal{F}^{\bot}$. Then the total extrinsic curvatures of $\mathcal{F}$ depend on $\kappa$, the volume of $M$ and the dimension of $\mathcal{F}$ only.
\end{cor}

Now, we show that \eqref{eq:sigmau_totgeodesic} implies the formula for $S_r$ obtained by Brito and Naveira \cite{BN}. Notice that
\begin{equation*}
\alpha^2_\flat(u)!=\frac{u!}{(u_{\alpha}-1)u_{\alpha}}\quad\textrm{and}\quad
\binom{\frac{r-2}{2}}{\frac{\alpha^2_\flat(u)}{2}}=\frac{u_{\alpha}}{r}\binom{\frac{r}{2}}{\frac{u}{2}}.
\end{equation*}
Hence, by \eqref{eq:sigmau_totgeodesic},
\begin{align*}
S_r &=\sum_{\stackrel{u\in 2\mathbb{N}(q)}{|u|=r}} \binom{\frac{r}{2}}{\frac{u}{2}}\binom{r}{u}^{-1}\,\sigma^M_u \\
&=\frac{\kappa(p-r+2)}{r}\sum_{\stackrel{u\in 2\mathbb{N}(q)}{|u|=r}} \binom{\frac{r}{2}}{\frac{u}{2}}\binom{r}{u}^{-1}\sum_{\alpha}\sigma^M_{\alpha^2_\flat(u)}\\
&=\frac{\kappa(p-r+2)}{r}\sum_{\stackrel{u\in 2\mathbb{N}(q)}{|u|=r}}\sum_{\alpha} \frac{r}{u_{\alpha}}\binom{\frac{r-2}{2}}{\frac{\alpha^2_\flat(u)}{2}}\frac{(u_{\alpha}-1)u_{\alpha}(\alpha^2_\flat(u))!}{r!}\sigma^M_{\alpha^2_\flat(u)}\\
&=\frac{\kappa(p-r+2)}{(r-1)r}\sum_{\stackrel{u\in 2\mathbb{N}(q)}{|u|=r}}\sum_{\alpha} 
\binom{\frac{r-2}{2}}{\frac{\alpha^2_\flat(u)}{2}}\frac{\alpha^2_\flat(u)!}{(r-2)!}(u_{\alpha}-1)\sigma^M_{\alpha^2_\flat(u)}.
\end{align*}
Notice that any multi--index $u\in 2\mathbb{N}(q)$ of length $r-2$ can be obtained from $q$ multi--indices $(1^\sharp)^2(u),\ldots,(q^\sharp)^2(u)$ of length $r$. Since
\begin{equation*}
\left((\alpha^\sharp)^2(u)\right)_{\alpha}=u_{\alpha}+2,
\end{equation*}
then (for $|u|=r-2$)
\begin{equation*}
\sum_{\alpha}\Big(\left((\alpha^\sharp)^2(u)\right)_{\alpha}-1\Big)=|u|+q=q+r-2.
\end{equation*}
Finally,
\begin{equation*}
S_r=\frac{\kappa(p-r+2)(q+r-2)}{(r-1)r}S_{r-2}.
\end{equation*}
This recurrence relation implies the formula \cite{AW2}
\begin{equation*}
\int_M S_r=\left\{\begin{array}{rl}
\binom{\frac{p}{2}}{\frac{r}{2}}\binom{q+r-1}{r}\binom{\frac{q+r-1}{2}}{\frac{r}{2}}^{-1}\kappa^{\frac{r}{2}}{\rm vol}(M) & \textrm{for $p$ even and $q$ odd} \\
2^r\left(\left(\frac{r}{2}\right)!\right)^{-1}\binom{\frac{q}{2}+\frac{r}{2}-1}{\frac{r}{2}}\binom{\frac{p}{2}}{\frac{r}{2}}\kappa^{\frac{r}{2}}{\rm vol}(M) & \textrm{for $p$ and $q$ even} \\
0 & \textrm{otherwise}
\end{array}\right.,
\end{equation*} 
which is the formula of Brito and Naveira \cite{BN}.

\subsection{Foliations}

Assume the distribution $D$ is integrable, hence induces foliation $\mathcal{F}$. Then the shape operators $A_{\alpha}$ are self--adjoint, i.e., $A^{\ast}_{\alpha}=A_{\alpha}$. Therefore, by Proposition \ref{Cor-adj-T_u} generalized Newton transformations $T_u$ are self--adjoint. Thus

\begin{cor}\label{Cor-intD}
Assume $M$ is closed. Then, for any $u\in\mathbb{N}(q)$, total extrinsic curvature $\sigma^M_u$ of a foliation $\mathcal{F}$ is of the form
\begin{equation*}
\begin{split}
|u|\sigma^M_u=\sum_{\alpha,\beta}\int_P &\Big(\tr(R_{\alpha,\beta}T_{\beta_\flat\alpha_\flat(u)})
+g({\rm div}_{E'}T_{\beta_\flat\alpha_\flat(u)},(\nabla_{e_{\alpha}}e_{\beta})^{\top})\\
&-g(H_{\mathcal{F}^{\bot}},T_{\beta_\flat\alpha_\flat(u)}(\nabla_{e_{\alpha}}e_{\beta})^{\top})
+\sum_{\gamma}g((T_{\beta_\flat\alpha_\flat(u)}
(\nabla_{e_{\alpha}}e_{\gamma})^{\top},(\nabla_{e_{\gamma}}e_{\beta})^{\top})\Big),
\end{split}
\end{equation*}
where $H_{\mathcal{F}^{\bot}}$ denotes the mean curvature vector of distribution $\mathcal{F}^{\bot}$. Moreover, the divergence ${\rm div}_{E'}T_u$ satisfies the recurrence relation
\begin{equation*}
{\rm div}_{E'}T_u=\sum_{\alpha}\left( \tr_{\mathcal{F}}(R(e_{\alpha},T_{\alpha_\flat(u)}\cdot)\cdot)^{\top}-A_{\alpha}({\rm div}_{E'}T_{\alpha_\flat(u)}) \right)
\end{equation*}
and can be expressed explicitly
\begin{equation*}
{\rm div}_{E'}T_u=\sum_{1\leq s\leq |u|}\sum_{\alpha_1,\ldots,\alpha_s}
(-1)^{s-1}A_{\alpha_1}\ldots A_{\alpha_{s-1}}\left(\tr_D R(e_{\alpha_s},T_{(\alpha_s)_\flat\ldots(\alpha_1)_\flat(u)}\cdot)\cdot \right)^{\top}.
\end{equation*}
\end{cor}

\subsection{Reduction to the case of one shape operator}

In \cite{Rov} the author considered the case of a distribution $D$ of arbitrary codimension with the only one shape operator $A_N$, where $N$ is a unit normal vector field to $D$. To get rid of the choice of $N$, the following operator is introduced
\begin{equation*}
A=\int_{S^{\bot}} A_N\,dN,
\end{equation*} 
where $S^{\bot}\subset D^{\bot}$ is the bundle of unit vectors orthogonal to $D$. This approach is similar to the one considered in this paper with a system of endomorphisms $(A_N,0,\ldots,0)$. Thus, we compute extrinsic curvatures of the form $\sigma^M_{(k,0,\ldots,0)}$, where $k\in\mathbb{N}$.

Put for simplicity
\begin{equation*}
T_{(k,0,\ldots,0)}=T_k,\quad \sigma_{(k,0,\ldots,0)}=\sigma_k,\quad e_1=N,\quad Z=(\nabla_NN)^{\top},\quad R_N=R_{N,N}.
\end{equation*}
Then, by Theorem \ref{Thm-main1}
\begin{align*}
k\sigma^M_k=\int_P &\big(\tr(R_NT_{k-2})+g({\rm div}_{E'}T^{\ast}_{k-2},Z)-g(H_{D^{\bot}},T_{k-2}Z)\\
&+\sum_{\gamma}g(T^{\ast}_{k-2}(\nabla_Ne_{\gamma})^{\top},(\nabla_{e_{\gamma}}N)^{\top})\big).
\end{align*}
Notice that $\sigma_k$, $T_k$ etc. depend only on $N$, i.e., $\sigma_k(x,e)=\sigma_k(N)$, $T_k(x,e)=T_k(N)$ etc. Let $H=O(q-1)$ (resp. $H=SO(q-1)$). Then $H$ is a closed subgroup of $G$ and $G/H=\mathbb{S}^{q-1}$, where $\mathbb{S}^{q-1}$ denotes the unit $(q-1)$--dimensional sphere. Moreover, the $G$--invariant measure on $G/H$ is just a Lebesgue measure $\lambda$ on the unit sphere. By Fubini theorem (see for example \cite{Hel})
\begin{align*}
\sigma^M_k &=\int_P \sigma_k(x,e)\,d(x,e)\\
&=\int_M \int_G \sigma_k(x,e_0g)\,dg\,dx\\
&=\int_M \int_{\mathbb{S}^{q-1}}\int_H \sigma_k(x,e_0g)\,dg\, d\lambda\,dx\\
&=\int_M\int_{\mathbb{S}^{q-1}}\sigma_k(N_{e_0})\,d\lambda(N)\,dx\\
&=\int_M\int_{S^{\bot}_x}\sigma_k(N)\,d\lambda(N)\,dx,
\end{align*}
where $e_0$ is a fixed basis of $D^{\bot}_x$, $N_{e_0}$ denotes the coordinates of $N$ with respect to basis $e_0$ and $S^{\bot}_x$ is the set of unit vectors in $D^{\bot}_x$.

In the case of a codimension one and the integrability of the distribution the formula for $\sigma^M_k$ gives the formula obtained by K. Andrzejewski and P. Walczak \cite{AW1}. Indeed, taking $G=SO(1)=\{1\}$, by above considerations we have
\begin{equation*}
k\sigma^M_k=\int_M\tr(R_NT_{k-2})+g({\rm div}^{\top}T_{k-2},Z),
\end{equation*}  
which is the formula \cite[Corollary 3.6]{AW1}.

\section{Appendix -- Differentiation and integration on principal bundles}

We derive useful formulas concerning the differential in the direction of a horizontal vector field on a principal bundle. Lets first recall some basic facts about principal bundles.

Let $\pi:P\to M$ be a principal fiber bundle with the structure group $G$. Let $\mathcal{V}={\rm ker}\pi_{\ast}$ be the vertical distribution. Let $\mathcal{H}$ be a horizontal distribution of a fixed connection on $P$. Then $\mathcal{H}$ is complementary to $\mathcal{V}$, i.e. $TP=\mathcal{V}\oplus\mathcal{H}$. For any vector $X\in T_xM$ and an element $u\in P_x=\pi^{-1}(x)$ there is unique horizontal vector $X^h_u\in\mathcal{H}_u$ called {\it the horizontal lift} of $X$.

Let $s\in\Gamma(P)$ be a section of $P$ and $f:P\to \mathbb{R}$ be a smooth function. Let $x\in M$ and assume $s$ is parallel at $x$, i.e.  $s_{\ast x}(T_xM)=\mathcal{H}_x$. Then, for every $X\in T_xM$
\begin{equation}\label{appx-pr1}
X^h_{s(x)}f=X(f\circ s)
\end{equation}
Indeed, it follows from the fact that for a parallel section $s$ at a point $x\in M$ we have $s_{\ast x}X=X^h_{s(x)}$.

Let $\pi_V:V\to M$ be a vector bundle with a fiber metric $g_V$ and the metric connection $\nabla^V$. Let $E=\pi^{-1}V\to P$ be the pull--back bundle, i.e. the bundle with a fiber $(\pi^{-1}V)_w=V_{\pi(w)}$, $w\in P$. There is a unique connection $\nabla^E$ in this bundle such that \cite{BW}
\begin{equation*}
\nabla^E_Z(X\circ\pi)=\left(\nabla^V_{\pi_{\ast}Z}X\right)\circ\pi,\quad X\in\Gamma(V),Z\in\Gamma(TP).
\end{equation*}

Let $s\in\Gamma(P)$ be a parallel section at a point $x\in M$ and let $Y\in\Gamma(E)$ be a section of a pull--back bundle $E$. Let $E_1,\ldots,E_m$ be a local basis in $V$ such that $\nabla^VE_a=0$ at $x\in M$. With respect to this basis $Y=\sum_a y_a(E_a\circ \pi)$ for some functions $y_a:P\to\mathbb{R}$. By \eqref{appx-pr1} we have
\begin{equation*}
\left( \nabla^E_{X^h}Y \right)\circ s=\sum_a\left(\left(X^hy_a\right)\circ s\right) E_a=\sum_a X(y_a\circ s)E_a=\nabla^V_X(Y\circ s).
\end{equation*}
Hence, for every $X\in T_xM$ and $s$ parallel at $x$
\begin{equation}\label{appx-pr2}
\left( \nabla^E_{X^h}Y \right)\circ s=\nabla^V_X(Y\circ s).
\end{equation}

Assume now that the structure group $G$ is compact and let $\lambda_G$ denotes the Haar measure on $G$. Let $f:P\to \mathbb{R}$ be a smooth function. We define the integral of $f$ over the fiber $P_x$, $x\in M$, as follows
\begin{equation*}
\int_{P_x}f(w)\,dw=\int_Gf(w_0g)d\lambda_G(g),
\end{equation*}
where $w_0\in P_x$ is fixed and $w_0g$ denotes the right multiplication in $P$ be element $g\in G$. By the invariance of the Haar measure, it follows that the integral is well defined.

Function $f:P\to \mathbb{R}$ induces a function $\widehat{f}:M\to\mathbb{R}$ by the formula
\begin{equation*}
\widehat{f}(x)=\int_{P_x}f(w)\,dw.
\end{equation*}

\begin{proposition}\label{appx-pr3}
Let $X\in T_xM$ and $f:P\to \mathbb{R}$ be a smooth function. Then the following formula holds
\begin{equation}\label{appx-int1}
X\widehat{f}=\widehat{X^hf},\quad\textrm{i.e.}\quad X\left( \int_{P_x}f(w)\,dw \right)=\int_{P_x}(X^hf)(w)\,dw.
\end{equation}
\end{proposition}
\begin{proof}
Fix $x\in M$ and let $X\in T_xM$. Let $\gamma$ be a curve on $M$ such that $\gamma(0)=x$ and $\gamma '(0)=X$. Moreover fix $w_0\in P_x$. Let $\gamma^h$ be a horizontal lift of $\gamma$ such that $\gamma^h(0)=w_0$. Then $(\gamma^h)'(0)=X^h_{w_0}$. Put $w_t=\gamma^h(t)$. Therefore,
\begin{align*}
X\widehat{f} &=\frac{d}{dt}(\widehat{f}\circ\gamma(t))_{t=0}\\
&=\frac{d}{dt}\left(\int_Gf(w_tg)\,d\lambda_G(g)\right)_{t=0}\\
&=\int_G\frac{d}{dt}\left(f(w_tg)\right)_{t=0}\,d\lambda_G(g).
\end{align*}
Furthermore
\begin{equation*}
\frac{d}{dt}\left(f(w_tg)\right)_{t=0}=\frac{d}{dt}\left(f(R_g(w_t))\right)_{t=0}
=(R_{g\ast}X^h_{w_0})f=X^h_{w_0g}f,
\end{equation*}
where $R_g(w)=wg$ is the right multiplication by $g\in G$. Hence
\begin{equation*}
X\widehat{f}=\int_G X^h_{w_0g}f\,d\lambda_G(g)=\int_{P_x}(X^hf)(w)\,dw=\widehat{X^hf}.
\qedhere
\end{equation*} 
\end{proof}

Let $Y\in\Gamma(E)$. Then $Y$ is a mapping of bundles $P$ and $V$ over the identity on $M$, i.e. $Y:P_x\to V_x$, $x\in M$. $Y$ induces a vector field $\widehat{Y}\in\Gamma(TM)$ as follows
\begin{equation*}
\widehat{Y}(x)=\int_{P_x}Y(w)\,dw,
\end{equation*}
where the integral of $Y$ is the integral of coordinates of $Y$ and is independent on the choice of point--wise basis in $V$. 

\begin{proposition}\label{appx-pr4}
Let $X\in T_xM$ and $Y\in\Gamma(E)$. Then the following formula holds
\begin{equation}\label{appx-int2}
\nabla^V_X\widehat{Y}=\widehat{\nabla^E_{X^h}Y},\quad\textrm{i.e.}\quad
\nabla^V_X\left( \int_{P_x}Y(w)\,dw \right)=\int_{P_x}\left( \nabla^E_{X^h}Y \right)\,dw.
\end{equation}
\end{proposition} 
\begin{proof}
Proof is similar to the proof of \eqref{appx-pr2}. Let $E_1,\ldots,E_m$ be a local basis in $V$ such that $\nabla^VE_a=0$ at $x\in M$. With respect to this basis $Y=\sum_a y_a(E_a\circ \pi)$ for some functions $y_a:P\to\mathbb{R}$. By Proposition \ref{appx-pr3} we have
\begin{equation*}
\nabla^V_X\widehat{Y}=\sum_a(X\widehat{y_a})E_a=\sum_a\widehat{X^hy_a}E_a=\widehat{\nabla_{X^h}Y}.
\qedhere
\end{equation*}
\end{proof}

Let $g_P$ be the Riemannian metric on $P$ induced from the Riemannian metric $g$ on $M$ and the invariant metric $\dbla\,,\,\dbra$ on the Lie algebra $\mathfrak{g}$ of the structure group $G$,
\begin{align*}
g_P(X^h,Y^h) &=g(X,Y),\\
g_P(X^y,A^{\ast}) &=0,\\
g_P(A^{\ast},B^{\ast}) &=\dbla A,B\dbra,
\end{align*}
where $X,Y$ are vectors on $M$ and $A^{\ast},B^{\ast}$ fundamental vertical vector fields induced by elements $A,B\in\mathfrak{g}$. Then the projection $\pi:P\to M$ is a Riemannian submersion,hence the Fubini theorem (see for example \cite{C})
\begin{equation}\label{Fubini}
\int_P f(w)\,dw=\int_M\left( \int_{P_x} f(w)\,dw \right)dx
\end{equation}
holds. 

Assume now that $V$ is a subbundle of a tangent bundle $TM$ and the metric $g_V$ is just a restriction of the Riemannian metric $g$. 

Let $Y\in\Gamma(E)$. The divergence ${\rm div}_EY$ of a section $Y$ is defined as follows
\begin{equation*}
{\rm div}_EY=\sum_a g(\nabla^E_{E_a}Y,\pi_{\ast}E_a),
\end{equation*}
where $(E_a)$ is a local orthonormal basis of $P$ which projects on $V$. Notice that the divergence ${\rm div}_EY$ can be written in the following way
\begin{equation}\label{appx-divfor}
{\rm div}_EY=\sum_ig(\nabla^E_{f^h_i}Y,f_i\circ\pi),
\end{equation}
where $(f_i)$ is a local orthonormal basis of $V$. By \eqref{appx-divfor} it follows that
\begin{equation}\label{appx-divfunc}
{\rm div}_E(\varphi Y)=\varphi{\rm div}_EY+Y^h\varphi,
\end{equation}
where $Y\in\Gamma(E)$ and $\varphi$ is a smooth function on $P$.

\begin{proposition}\label{appx-pr5}
The divergence ${\rm div}_EY$ of a section $Y\in\Gamma(E)$ and the divergence ${\rm div}_V\widehat{Y}$ of a vector field $\widehat{Y}\in\Gamma(V)$ are related as follows
\begin{equation*}
\int_P {\rm div}_EY\,dw=\int_M {\rm div}_V\widehat{Y}\,dx.
\end{equation*}
In particular, if $M$ is closed and $V=TM$, then $\int_P{\rm div}_EY\,dw=0$.
\end{proposition}
\begin{proof}
By Proposition \ref{appx-pr4} we have
\begin{equation*}
\widehat{{\rm div}_EY}=\sum_ig(\widehat{\nabla^E_{f_i^h}Y},f_i)=\sum_ig(\nabla^V_{f_i}\widehat{Y},f_i)={\rm div}_V\widehat{Y}.
\end{equation*}
Thus by Fubini theorem \eqref{Fubini}
\begin{equation*}
\int_P{\rm div}_EY\,dw=\int_M\widehat{{\rm div}_EY}\,dx=\int_M {\rm div}_V\widehat{Y}\,dx.
\qedhere
\end{equation*}
\end{proof}

Let ${\rm End}(E)$ denote the bundle of endomorphisms of $E$, i.e., the fiber ${\rm End}(E)_x$, $x\in M$, is ${\rm End}(V_x)$. The connection $\nabla^E$ induces the connection in ${\rm End}(E)$, namely
\begin{equation*}
(\nabla^E_ZS)W=\nabla^E_Z(SW)-S(\nabla^E_ZW),\quad S\in\Gamma({\rm End}(E)), Z\in\Gamma(P),W\in\Gamma(E).
\end{equation*}
Then, the divergence ${\rm div}_ES$ of $S\in\Gamma({\rm End}(E))$ is defined as follows
\begin{equation*}
{\rm div}_ES=\sum_{E_a}(\nabla^E_{E_a}S)(\pi_{\ast}E_a).
\end{equation*}  
where $(E_a)$ is a local orthonormal basis of $P$ which projects on $V$. Notice that ${\rm div}_ES$ can be written in the form
\begin{equation*}
{\rm div}_ES=\sum_i(\nabla^E_{f_i^h}S)(f_i\circ\pi),
\end{equation*}
where $(f_i)$ is a local orthonormal basis of $V$.
\begin{proposition}\label{appx-pr6}
Let $s\in\Gamma(P)$ be a parallel section at a point $x\in M$, let $S\in\Gamma({\rm End}(E))$ and $X\in T_xM$. Then
\begin{equation*}
\left(\nabla^E_{X^h}S\right)(s(x))=\left(\nabla^V_X(S\circ s)\right)(x)
\end{equation*}
and
\begin{equation*}
\left({\rm div}_ES\right)(s(x))=\left({\rm div}_V(S\circ s)\right)(x).
\end{equation*}
\end{proposition}
\begin{proof}
Follows immediately by \eqref{appx-pr2}.
\end{proof}

The following formulas hold
\begin{align}
{\rm div}_E(\varphi S) &=\varphi{\rm div}_ES+S(\pi_{\ast}{\rm grad}_{V^h}\varphi),\label{divE1}\\
{\rm div}_E(TS) &=\sum_i(\nabla^E_{f_i^h}T)(Sf_i)+T({\rm div}_ES),\label{divE2}
\end{align}
for $S,T\in\Gamma({\rm End}(E))$, $\varphi:P\to\mathbb{R}$, where $(f_i)$ is a local orthonormal basis of $V$ and ${\rm grad}_{V^h}\varphi$ denotes the $V^h$--component of the gradient of a function $\varphi$.

\begin{Ackn}
The authors would like to thank Pawe{\l} Walczak and Szymon M. Walczak for helpful conversations and support.
\end{Ackn}

\end{document}